\documentclass[reqno]{amsart}
\usepackage{yhmath,amsmath,amsfonts,amssymb,enumerate,amsthm,appendix,caption,lipsum,xparse}
\usepackage{enumitem}
\usepackage{mathrsfs}
\usepackage[english]{babel}
\usepackage{cmap,mathtools}
\usepackage{graphics} 
\usepackage{epsfig} 
\usepackage{graphicx}\usepackage{epstopdf}
\usepackage[pagewise]{lineno}
\newcommand\numberthis{\addtocounter{equation}{1}\tag{\theequation}}
\usepackage{a4wide}
\usepackage[margin=0.95in]{geometry}
\usepackage{cite}
\usepackage[utf8]{inputenc}
\usepackage{xcolor}
\usepackage{hyperref}
\usepackage[hyperpageref]{backref}
\hypersetup{
    colorlinks=true,
    linkcolor=myblue,
    filecolor=magenta,      
    urlcolor=mygreen,
    citecolor=mygreen,
}
\definecolor{mygreen}{rgb}{0.01,0.6,0.2}
\definecolor{myblue}{rgb}{0.01, 0.18, 1.0}
\newtheorem{theorem}{Theorem}
\newtheorem{proposition}[theorem]{Proposition}
\newtheorem{lemma}[theorem]{Lemma}

\theoremstyle{definition}
\newtheorem{definition}[theorem]{Definition}
\newtheorem{remark}[theorem]{Remark}
\newtheorem{example}[theorem]{Example}

\numberwithin{equation}{section}
\numberwithin{theorem}{section}
\numberwithin{equation}{section}
\numberwithin{theorem}{section}
\title[Existence and non-existence results]{Existence and non-existence results to a mixed Schrodinger system in a plane}
 \author[H.Hajaiej, R. Kumar, T. Mukherjee \&L. Song]{HICHEM HAJAIEJ$^{1}$, Rohit Kumar$^2$, Tuhina Mukherjee$^{2}$ and Linjie Song$^3$}
 \thanks{$^*$Corresponding author.}
\begin{document}
\def\dxy{\mathrm{d}x \mathrm{d}y}
\def\d{\mathrm{d}}
\def\dxz{\mathrm{d}x \mathrm{d}z}
\def\dyz{\mathrm{d}y \mathrm{d}z}
\def\R{\mathbb{R}}
\def\({\bigg(}
\def\){\bigg)}
\def\H{\mathcal{H}^{1,s}(\R^2)}
\def\Hi{\mathcal{H}^{1,s_i}(\R^2)}
\def\Hone{\mathcal{H}^{1,s_1}(\R^2)}
\def\Htwo{\mathcal{H}^{1,s_2}(\R^2)}
\def\ui{\int_{\R^2} \big(|u|^2 + |\partial_x u|^2 + |(-\Delta)^{s_{i}/2}_y u|^2 \big) \dxy}
\def\uone{|u|^2 + |\partial_x u|^2 + |(-\Delta)^{s_{1}/2}_{y} u|^2 }
\def\unone{|u_n|^2 + |\partial_x u_n|^2 + |(-\Delta)^{s_{1}/2}_{y} u_n|^2 }
\def\v2{\int_{\R^2} \big(|v|^2 + |\partial_x v|^2 + |(-\Delta)^{s_{2}/2}_yv|^2 \big) \dxy}
\def\uq{\(\int_{\R^2}|u|^q \dxy\)^{2/q}}
\def\u2si{\(\int_{\R^2}|u|^{2_{s_i}} \dxy\)^{2/2_{s_i}}}
\def\Dely{(-\Delta)_{y}^{s/2}}
\def\Delone{(-\Delta)_{y}^{s_1/2}}
\def\Deltwo{(-\Delta)_{y}^{s_1/2}}
\def\C_c{C_{c}^{\infty}(\R^2)}
\definecolor{dogwoodrose}{rgb}{0.84, 0.09, 0.41}
\definecolor{lava}{rgb}{0.81, 0.06, 0.13}
\def\tb{\color{blue}}
\definecolor{ao(english)}{rgb}{0.0, 0.5, 0.0}
\def\ao{\color{ao(english)}}
\definecolor{robineggblue}{rgb}{0.0, 0.8, 0.8}
\def\tro{\color{tiffanyblue}}
\definecolor{tiffanyblue}{rgb}{0.04, 0.73, 0.71}
\definecolor{tenné(tawny)}{rgb}{0.8, 0.34, 0.0}
\def\ten{\color{tenné(tawny)}}
\definecolor{aqua}{rgb}{0.0, 1.0, 1.0}
\def\aqua{\color{aqua}}
\maketitle 
\centerline{$^{1}$Department of Mathematics, California State University at Los Angeles}
\centerline{Los Angeles, CA 90032,
USA} 
 \centerline{$^{2}$Department of Mathematics, Indian Institute of Technology Jodhpur,}
 \centerline{Rajasthan 342030, India}
\centerline{$^{3}$Department of Mathematical Sciences, Tsinghua University, Beijing 100084} 
 \begin{abstract}
  This article focuses on the existence and non-existence of solutions for the following system of local and nonlocal type
  \begin{equation*}
    \left\{
     \begin{aligned}
         -\partial_{xx}u + (-\Delta)_{y}^{s_{1}} u + u - u^{2_{s_{1}}^{}-1} = \kappa \alpha h(x,y) u^{\alpha-1}v^{\beta} & \quad \mbox{in} ~ \mathbb{R}^{2}, \\
            -\partial_{xx}v + (-\Delta)_{y}^{s_{2}} v + v- v^{2_{s_{2}}^{}-1} = \kappa \beta h(x,y) u^{\alpha}v^{\beta-1} & \quad \mbox{in} ~ \mathbb{R}^{2},\\
             u,v ~ \geq  ~0 \quad \mbox{in} ~ \mathbb{R}^{2},
    \end{aligned}
    \right.
\end{equation*}
 where $s_{1},s_{2} \in (0,1),~\alpha,\beta>1,~\alpha+\beta \leq \min \{ 2_{s_{1}}^{},2_{s_{2}}^{}\}$, and $2_{s_i}^{} = \frac{2(1+s_i)}{1-s_i}, i=1,2$. The existence of a ground state solution entirely depends on the behaviour of the parameter $\kappa>0$ and on the function $h$. In this article, we prove that a ground state solution exists in the subcritical case if $\kappa$ is large enough and $h$ satisfies \eqref{condition on h}.  Further, if $\kappa$ becomes very small in this case then there does not exist any solution to our system. The study in the critical case, i.e. $s_1=s_2=s, \alpha+\beta=2_s$, is more complex and the solution exists only for large $\kappa$ and radial $h$ satisfying \eqref{H one}. Finally, we establish a Pohozaev identity which enables us to prove the non-existence results under some smooth assumptions on $h$.
 \end{abstract} 
\section{Introduction}\label{S1}
\noindent In this article, we want to study the existence and non-existence of solutions for the following system associated with mixed Schrodinger operator: 
\begin{equation} \label{main problem}
     \left\{
     \begin{aligned}
         -\partial_{xx}u + (-\Delta)_{y}^{s_{1}} u + u - u^{2_{s_{1}}^{}-1} = \kappa \alpha h(x,y) u^{\alpha-1}v^{\beta} & \quad \mbox{in} ~ \mathbb{R}^{2}, \\
            -\partial_{xx}v + (-\Delta)_{y}^{s_{2}} v + v- v^{2_{s_{2}}^{}-1} = \kappa \beta h(x,y) u^{\alpha}v^{\beta-1} & \quad \mbox{in} ~ \mathbb{R}^{2},\\
            u,v ~ \geq ~ 0 \quad \mbox{in} ~ \mathbb{R}^{2},
    \end{aligned}
    \right.
\end{equation}
where $s_{1},s_{2} \in (0,1)$. The exponent $2_{s_i}^{} = \frac{2(1+s_i)}{1-s_i},(i=1,2)$ plays a vital role in the existence and non-existence of solutions to the system \eqref{main problem} and therefore we call it as a critical exponent. The parameter $\kappa>0$ and $\alpha ,\beta$ are chosen such that 
\begin{align} \label{ alpha beta condition}
    \alpha,\beta>1~~\text{and}~~\alpha+\beta \leq \min \{ 2_{s_{1}}^{},2_{s_{2}}^{}\}.
\end{align}
Throughout this paper we always assume that the function $h$ satisfies the following condition 
 \begin{align} \label{condition on h}
    0 ~ \leq ~ h \in L^{1}(\mathbb{R}^{2})\cap L^{\infty}(\mathbb{R}^{2}), h \text{ is not zero}.
 \end{align}
 To study the motivation and the applications of the mixed Schrodinger operator $\mathcal{L}:= -\partial_{xx} + (-\Delta)_{y}^{s}  + I$ in detail, we refer to the article \cite{Esfahani2018} and the references therein. In 2018, Amin Esfahani et. al. \cite{Esfahani2018} dealt with the following problem 
 \begin{equation} \label{Esfahani problem}
     u + (- \Delta_x)^su-\Delta_y u=f(u), \quad (x,y)\in \R^N \times \R^M, 
 \end{equation}
 where $N,M \geq 1$ and the operator $(- \Delta_x)^s, s\in (0,1)$ denotes the fractional Laplacian in $x$-variable, which is defined (on appropriate functions) in the Cauchy principle value (p.v.) sense up to a normalizing constant $C_{N,s}$  by
 \[(- \Delta_x)^s u(x,y)= \text{p.v.} \int_{\R^N} \frac{u(x,y)-u(z,y)}{|x-z|^{N+2s}}\,\d z. \]
 The function $f \in C^1(\R,\R)$ satisfies the subcritical growth assumptions without Ambrosetti-Rabinowitz type condition. The equation \eqref{Esfahani problem} admits a positive ground state radial solution  (see \cite[Theorem 1.3]{Esfahani2018}). In 2019, Felmer and Wang \cite{Felmer2019} also considered the problem of type \eqref{Esfahani problem} and studied the qualitative properties of positive solutions under different assumptions on $f$.
 Recently, Gou et. al.\cite{Hichem2023} provided a comprehensive study about the existence and non-existence of solutions of the following problem on a plane
\begin{equation}\label{Single problem with p}
   -\partial_{xx}u + (-\Delta)_{y}^{s_{}} u + u  = u^{p -1} \text{  in  } \R^2.
\end{equation}
The authors showed that the problem \eqref{Single problem with p} has a positive ground state solution which is axially symmetric when $2 < p < 2_s$, and by using Pohozaev identity it has only trivial solution for $p \geq 2_s$.
If $\kappa=0$, the system \eqref{main problem} reduces to a single equation of the form
 \begin{equation} \label{single problem}
    -\partial_{xx}u + (-\Delta)_{y}^{s_{}} u + u  = u^{2_s -1} \quad  \text{in}~ \mathbb{R}^2,
\end{equation}
which clearly does not possess any non-trivial solution  by taking $p=2_s$ in \eqref{Single problem with p}. This shows that the system \eqref{main problem} has no semi-trivial solutions, i.e. solutions like $(u,0)$ or $(0,v)$. In this paper, we are interested in the effects of $\kappa$ and $h(x,y)$ on the existence of fully non-trivial solutions (particularly, ground state solutions, whose definition will be given later) for the system \eqref{main problem}, i.e. solutions $(u,v)$ with $u \neq 0$ and $v \neq 0$. Since the equation \eqref{Single problem with p} has non-trivial solutions when $2 < p < 2_s$ (see \cite{Hichem2023}), when $\alpha + \beta < \min \{ 2_{s_{1}}^{},2_{s_{2}}^{}\}$, we expect that a large $\kappa$ and positive $h$ will bring the existence of ground state solutions for \eqref{main problem}. Indeed, as readers will see, our expection is right. Further, this is right when $\alpha + \beta = \min\{2_{s_1}, 2_{s_2}\} < \max\{2_{s_1}, 2_{s_2}\}$. Situations become more complex in the critical case that $s_1 = s_2 = s, \alpha +\beta = 2_s$.
On the one hand, we can prove that a solution exists under a setting that $\kappa$ is large and $h$ is radial and satisfies the following assumption \eqref{H one}:
	\begin{equation}\label{H one}
		0 \leq h \in L^{1}(\R^2) \cap L^{\infty}(\R^2)
		,~h~\text{is not zero, continuous~near}~0~\text{and}~\infty,~\text{and}~h(0,0) = \lim\limits_{|(x,y)| \rightarrow +\infty}h(x,y) = 0. \tag\mathbb{H1}
	\end{equation}
On the other hand, the system \eqref{main problem} has no non-trivial solutions for some $h$, such as $h$ being a positive constant, and more examples of $h$ will be given later. Our study opens a door to understand that the conditions on $\kappa$ and $h$ are crucial for the existence and non-existence of non-trivial solution for \eqref{main problem}.

 To state our main results,
we first define the fractional Sobolev-Liouville space $\mathcal{H}^{1,s_{i}}(\mathbb{R}^{2}),~ i=1,2$ (see \cite{Esfahani2015}). The space $\mathcal{H}^{1,s_{i}}(\mathbb{R}^{2})$ is a Banach space with respect to the norm
\[ \|u\|_{\mathcal{H}^{1,{s_i}}}^2 := \|u\|_{2}^{2} +\|\partial_x u\|_{2}^{2}+\|(-\Delta)_{y}^{s_i /2} u\|_{2}^{2}   ,\text{ for }i=1,2.\]
The product space $\mathbb{D} = \mathcal{H}^{1,s_{1}}(\mathbb{R}^{2}) \times \mathcal{H}^{1,s_{2}}(\mathbb{R}^{2})$ is associated with the following norm
\[ \|(u,v)\|^{2}_{\mathbb{D}} = \|u\|_{\mathcal{H}^{1,{s_1}}}^2 + \|v\|_{\mathcal{H}^{1,{s_2}}}^2.  \]
The energy functional associated with the system \eqref{main problem} is defined by $J_\kappa: \mathbb{D} \rightarrow \mathbb{R}$ and given as follows:
\begin{align} \label{energy functional}
	\begin{split}
		J_{\kappa}(u,v) &= \frac{1}{2}\int_{\mathbb{R}^{2}} \big( |u|^2 +|\partial_{x}u|^2 +  |(-\Delta)_{y}^{s_1 /2} u|^2 \big)\,\mathrm{d}x \mathrm{d}y +\frac{1}{2}\int_{\mathbb{R}^{2}} \big( |v|^2 +|\partial_{x}v|^2 +  |(-\Delta)_{y}^{s_2 /2} v|^2 \big)\,\mathrm{d}x \mathrm{d}y  \\
		&\quad - \frac{1}{2_{s_{1}}}\int_{\mathbb{R}^{2}} |u|^{2_{s_{1}}^{}}\, \mathrm{d}x \mathrm{d}y -\frac{1}{2_{s_{2}}^{}}\int_{\mathbb{R}^{2}} |v|^{2_{s_{2}}^{}}\, \mathrm{d}x \mathrm{d}y  - \kappa \int_{\mathbb{R}^{2}}h(x,y)|u|^{\alpha}|v|^{\beta}\, \mathrm{d}x \mathrm{d}y.
	\end{split}
\end{align}  
The above functional $J_\kappa$ is well-defined and $C^1$ on the product space $\mathbb{D}$. For  $(u_0,v_0) \in \mathbb{D}$, the Fr\'{e}chet derivative of $J_\kappa$ at $(u,v) \in \mathbb{D}$ can be represented as
\begin{align*}
	\langle J'_{\kappa}(u,v) | (u_0,v_0)\rangle &= \int_{\mathbb{R}^{2}}  \big( u  u_0 + \partial_x u  \partial_x u_0 +(-\Delta)_{y}^{s_1/2}u  (-\Delta)_{y}^{s_1/2}u_0  \big)\,\mathrm{d}x\mathrm{d}y \\
	&\qquad + \int_{\mathbb{R}^{2}}  \big( v  v_0 + \partial_x v  \partial_x v_0 +(-\Delta)_{y}^{s_2/2}v  (-\Delta)_{y}^{s_2/2}v_0  \big)\,\mathrm{d}x\mathrm{d}y  \\ 
	& \quad- \int_{\mathbb{R}^{2}} |u|^{{2_{s_{1}}}-2} u  u_0 \,\mathrm{d}x\mathrm{d}y
	- \int_{\mathbb{R}^{2}} |v|^{{2_{s_{2}}}-2} v  v_0 \,\mathrm{d}x\mathrm{d}y\\
	&~~~~~ - \kappa \alpha \int_{\mathbb{R}^{2}}h(x,y)|u|^{\alpha -2}u u_0|v|^{\beta}\,\mathrm{d}x\mathrm{d}y
	- \kappa \beta \int_{\mathbb{R}^{2}}h(x,y)|u|^{\alpha}|v|^{\beta-2}v  v_0 \,\mathrm{d}x\mathrm{d}y,
\end{align*}
where $J'_{\kappa}(u,v)$ is the Fr\'{e}chet derivative of $J_\kappa$ at $(u,v)$, and the duality bracket between the product space $\mathbb{D}$ and its dual $\mathbb{D}^*$ is represented as $\prescript{}{\mathbb{D}^*}{\langle}  \cdot,\cdot \rangle_{\mathbb{D}}$. We just write $\langle \cdot, \cdot \rangle$ if no confusion of notation arises. 
\begin{definition}[\textbf{Ground state solution}] \label{def of gss}
	 A non-trivial critical point $(u_1,v_1) \in \mathbb{D} \backslash \{(0, 0)\}$ of $J_\kappa$ over $\mathbb{D}$ is said to be a ground state solution if its energy is minimal among all the non-trivial critical points i.e.
	\begin{equation} \label{ground state level}
		c_{\kappa} = J_{\kappa}(u_{1},v_{1}) = \min \{ J_{\kappa}(u,v): (u, v) \in \mathbb{D} \setminus \{(0, 0)\} ~\text{and}~ J'_{\kappa}(u,v)=0 \}.
	\end{equation}  
\end{definition} 
\noindent As explained in Section \ref{secpre} below, the functional $J_\kappa$ is not bounded from below on the product space $\mathbb{D}$. For this reason, we introduce the definition of the Nehari manifold $\mathcal{N}_{\kappa}$ associated with the functional $J_\kappa$ as
	\[\mathcal{N}_{\kappa} = \{  (u, v) \in \mathbb{D} \backslash \{(0, 0)\} : \Phi_{\kappa}(u,v) = 0    \},\]
	where 
	\begin{equation} \label{phi function}
		\Phi_{\kappa}(u,v) = \langle J'_{\kappa}(u,v) | (u,v) \rangle.
	\end{equation}
Define $\Bar{c}_\kappa := \inf\limits_{(u,v) \in \mathcal{N}_{\kappa}}J_{\kappa}(u,v)$. Notice that the Nehari manifold $\mathcal{N}_\kappa$ consists of all the non-trivial critical points of the functional $J_\kappa$. Hence, if $\Bar{c}_\kappa$ can be achieved by some $(u,v)$, then $(u,v)$ is a ground state solution.

\begin{proposition} \label{prop1}
We have the following conclusions:
\begin{itemize}
	\item[(1)] $\Bar{c}_\kappa$ is non-increasing w.r.t. $\kappa \geq 0$, and further, $\Bar{c}_\kappa > \Bar{c}_{\kappa'}$ for all $\kappa' > \kappa$ if  $\Bar{c}_\kappa$ can be achieved;
	\item[(2)] $\Bar{c}_\kappa$ is continuous w.r.t. $\kappa \in [0,\infty)$;
	\item[(3)] there exists $\kappa^* \in [0,\infty)$ such that 
	$$
	\Bar{c}_\kappa = \min\bigg\{ \frac{s_1}{1+s_1} \Lambda_1^{\frac{1+s_1}{2 s_1}}, \frac{s_2}{1+s_2} \Lambda_2^{\frac{1+s_2}{2s_2}}  \bigg\}, \forall \kappa \in [0,\kappa^*],
	$$
	$$
	\Bar{c}_\kappa < \min\bigg\{ \frac{s_1}{1+s_1} \Lambda_1^{\frac{1+s_1}{2s_1}}, \frac{s_2}{1+s_2} \Lambda_2^{\frac{1+s_2}{2s_2}}  \bigg\}, \forall \kappa > \kappa^*,
	$$
	where $\Lambda_{i}~ (i = 1,2)$ will be defined by \eqref{Best constant}.
\end{itemize}
\end{proposition}	

\begin{remark}
	It is open whether $\kappa^* = 0$ or not.
\end{remark}
We are firstly concerned with the case that $2< \alpha +\beta < \min\{2_{s_{1}}, 2_{s_{2}}\}$ or $\alpha + \beta = \min\{2_{s_1}, 2_{s_2}\} < \max\{2_{s_1}, 2_{s_2}\}$.
\begin{theorem} \label{solution for kappa large}
	Let $2< \alpha +\beta < \min\{2_{s_{1}}, 2_{s_{2}}\}$ or $\alpha + \beta = \min\{2_{s_1}, 2_{s_2}\} < \max\{2_{s_1}, 2_{s_2}\}$. Recall that $\kappa^*$ is given by Proposition \ref{prop1}. Then we have:
	\begin{itemize}
		\item[(1)] for any $\kappa \in (\kappa^*,\infty)$, $\Bar{c}_\kappa$ has a minimizer, which is a non-negative (constant sign) ground state solution to the system \eqref{main problem};
		\item[(2)] for any $\kappa \in [0,\kappa^*)$, $\Bar{c}_\kappa$ has no minimizers.
	\end{itemize}
\end{theorem}

Next we focus on the critical case that $s_1 = s_2 = s, \alpha +\beta = 2_s$, which is more complex. 

\begin{theorem} \label{existence}
Suppose $s_1 = s_2 = s, \alpha +\beta = 2_s$ and that $h$ is radial satisfying \eqref{H one}. Then there is $\kappa^{**}$ (which will be given by Proposition \ref{prop2}) such that for any $\kappa \in (\kappa^{**},\infty)$, there exists a non-negative (constant sign) non-trivial solution to the system \eqref{main problem}.   
\end{theorem}
\begin{remark}
    It is still open whether the solution exists or not in the critical case for $\kappa>0$ small enough.
\end{remark}
To prove Theorem \ref{existence}, we find solutions in radial space $\mathbb{D}_r$ where
\[ \mathbb{D}_r := \mathcal{H}^{1,s}_r(\R^2) \times \mathcal{H}^{1, s}_r(\R^2) = \{(u,v) \in \mathbb{D} : u ~\text{and}~v~\text{are radially symmetric}\}. \]
We further define
\[\mathcal{N}_{\kappa,r} = \{  (u, v) \in \mathbb{D}_r \backslash \{(0, 0)\} : \Phi_{\kappa}(u,v) = 0    \},\]
where 
\begin{equation}
\Phi_{\kappa}(u,v) = \langle J'_{\kappa}(u,v) | (u,v) \rangle.
\end{equation}
Let $\Bar{c}_{\kappa,r} := \inf\limits_{(u,v) \in \mathcal{N}_{\kappa,r}}J_{\kappa}(u,v)$, and
\begin{equation} \label{Best constant radial}
	\Lambda_{r} := \inf\limits_{u \in \mathcal{H}^{1,s}_r(\R^2)}\frac{\|u\|_{\mathcal{H}^{1,s}(\R^2)}^2}{\|u\|_{2_{s}}^2}.
\end{equation}
 Similar to Proposition \ref{prop1}, we have the following result.

\begin{proposition} \label{prop2}
	Under the assumptions of Theorem \ref{existence}, we have the following conclusions:
	\begin{itemize}
		\item[(1)] $\Bar{c}_{\kappa,r}$ is non-increasing w.r.t. $\kappa \geq 0$, and further, $\Bar{c}_{\kappa,r} > \Bar{c}_{\kappa',r}$ for all $\kappa' > \kappa$ if $\bar{c}_{\kappa,r}$ can be achieved;
		\item[(2)] $\Bar{c}_{\kappa,r}$ is continuous w.r.t. $\kappa \in [0,\infty)$;
		\item[(3)] there exists $\kappa^{**} \in [0,\infty)$ such that 
		$$
		\Bar{c}_{\kappa,r} = \frac{s}{1+s} \Lambda_r^{\frac{1+s}{2s}}, \forall \kappa \in [0,\kappa^{**}],
		$$
		$$
		\Bar{c}_{\kappa,  r} < \frac{s}{1+s} \Lambda_r^{\frac{1+s}{2s}}, \forall \kappa > \kappa^{**}.
		$$
	\end{itemize}
\end{proposition}
On the other hand, for some $h$, \eqref{main problem} has no non-trivial solutions.

\begin{theorem}[\textbf{Non-existence of non-trivial solutions}] \label{Non-existence theorem}
 Let $s_1 = s_2 = s, \alpha +\beta = 2_s$. If $h$ and $\kappa$ satisfy either of the following hypothesis:
 \begin{itemize}
     \item[(a)] $h$ is a positive constant and $\kappa$ is any real number,
     \item[(b)] $h \in L^{\infty}(\R^2)\cap C^1(\R^2)$ such that $\kappa xh_x \leq 0$ and $\kappa yh_y \leq 0$,
 \end{itemize}
  then the system \eqref{main problem} has no non-trivial solution $(u,v)$ {which satisfies $xu,yu \in L^2(\R^2)$}.
\end{theorem} 
 We organise the rest of our paper as follows. The section \ref{secpre} contains some properties of the Nehari manifold and boundedness of the (PS) sequences. Further, the fractional Laplacian w.r.t $y$ variable is defined and some useful results associated with it are shown. In section \ref{concentration sec}, the concentration-compactness tools for the mixed Schrodinger operator are developed and by using this tool we prove the (PS) compactness condition of the functional $J_\kappa$ in the subcritical case. In section \ref{prop of ck}, some properties of $\Bar{c}_k$ and the proof of Theorem \ref{solution for kappa large} are given. In section \ref{proof sec6}, we give the proof of Theorem \ref{existence}. Finally, we establish the Pohozaev identity in section \ref{pohozaev sec} and using this Pohozaev identity we are able to give the non-existence result (see Theorem \ref{Non-existence theorem}), and conclude this section by giving some specific examples of $h$ for which the non-trivial solution does not exist.

\section{Preliminaries} \label{secpre}

\subsection{The Nehari manifold and (PS) sequences}

For any $\eta>0$ and \eqref{energy functional}, we can write
\begin{align} 
     J_{\kappa}(\eta u, \eta v) &= \frac{\eta^2}{2}\int_{\mathbb{R}^{2}} \big( |u|^2 +|\partial_{x}u|^2 +  |(-\Delta)_{y}^{s_1 /2} u|^2 \big)\,\mathrm{d}x \mathrm{d}y +\frac{\eta^2}{2}\int_{\mathbb{R}^{2}} \big( |v|^2 +|\partial_{x}v|^2 +  |(-\Delta)_{y}^{s_2 /2} v|^2 \big)\,\mathrm{d}x \mathrm{d}y  \\
    &\quad - \frac{\eta^{2_{s_1}}}{2_{s_{1}}}\int_{\mathbb{R}^{2}} |u|^{2_{s_{1}}^{}}\, \mathrm{d}x \mathrm{d}y -\frac{\eta^{2_{s_2}}}{2_{s_{2}}^{}}\int_{\mathbb{R}^{2}} |v|^{2_{s_{2}}^{}}\, \mathrm{d}x \mathrm{d}y  - \kappa \eta^{\alpha+\beta} \int_{\mathbb{R}^{2}}h(x,y)|u|^{\alpha}|v|^{\beta}\, \mathrm{d}x \mathrm{d}y.
\end{align} 
Observe that the functional $J_{\kappa}(\eta u_{},\eta v_{}) \rightarrow -\infty$ as $\eta$ becomes very large. Since the functional $J_\kappa$ is not bounded from below on the product space $\mathbb{D}$, we try to minimize the given functional on the Nehari manifold $\mathcal{N}_{\kappa}$ for finding a critical point in $\mathbb{D}$ using variational approach. 

If we consider  $(u,v) \in \mathcal{N}_{\kappa}$, then the following identity holds
\begin{align} \label{equivalent norm}
\begin{split}
    \|(u,v)\|_{\mathbb{D}}^{2} 
    = \|u\|_{{2_{s_{1}}^{}}}^{2_{s_{1}}^{}} + \|v\|_{{2_{s_{2}}^{}}}^{2_{s_{2}}^{}} + \kappa (\alpha + \beta) \int_{\mathbb{R}^{2}} h(x,y)|u|^{\alpha}|v|^{\beta}\mathrm{d}x \mathrm{d}y.
\end{split}
\end{align}
The restriction of the functional $J_\kappa$ on $\mathcal{N}_{\kappa}$ can be written as
\begin{align} \label{energy functional on Nehari manifold}
    J_{\kappa}|_{\mathcal{N}_{\kappa}}(u,v) = \frac{s_{1}}{1+s_1} \|u\|_{{2_{s_{1}}^{}}}^{2_{s_{1}}^{}} + \frac{s_{2}}{1+s_2} \|v\|_{{2_{s_{2}}^{}}}^{2_{s_{2}}^{}} + \kappa \bigg(  \frac{\alpha+\beta -2}{2}  \bigg) \int_{\mathbb{R}^{2}} h(x,y)|u|^{\alpha}|v|^{\beta}\mathrm{d}x \mathrm{d}y.
\end{align}
Further, if $(\eta u,\eta v) \in \mathcal{N}_{\kappa}$ for all $(u, v) \in \mathbb{D} \backslash \{(0, 0)\}$, using the identity \eqref{equivalent norm} we can obtain an algebraic equation of the form
\begin{align} \label{Algebraic equation}
     \|(u,v)\|_{\mathbb{D}}^{2} =  \eta^{2_{s_{1}} -2} \|u\|_{{2_{s_{1}}^{}}}^{2_{s_{1}}^{}} +  \eta^{2_{s_{2}} -2} \|v\|_{{2_{s_{2}}^{}}}^{2_{s_{2}}^{}} + \kappa (\alpha+\beta ) \eta^{\alpha+ \beta -2} \int_{\mathbb{R}^{2}} h(x,y)|u|^{\alpha}|v|^{\beta}\mathrm{d}x \mathrm{d}y.
\end{align}
By following a careful analysis of the above algebraic equation, it has a unique positive solution. Hence, there is a unique $\eta= \eta_{(u,v)}>0$ such that $(\eta u,\eta v) \in \mathcal{N}_{\kappa}$ for all $(u, v) \in \mathbb{D} \setminus \{(0, 0)\}$. For any $(u,v) \in \mathcal{N}_{\kappa}$, the identity (\ref{equivalent norm}) and the assumption (\ref{ alpha beta condition}) altogether gives
\begin{align*} 
       J''_{\kappa}(u,v)[u,v]^{2} &= \langle \Phi'_{\kappa}(u,v) | (u,v) \rangle  \\
    &= 2 \|(u,v)\|_{\mathbb{D}}^{2}-2_{s_{1}} \|u\|_{2_{s_{1}}^{}}^{2_{s_{1}}^{}} 
    -2_{s_{2}} \|v\|_{2_{s_{2}}^{}}^{2_{s_{2}}^{}} - \kappa (\alpha+\beta)^{2} \int_{\mathbb{R}^{2}}h(x,y)|u|^{\alpha}|v|^{\beta}\,\dxy\\
       & =(2-\alpha-\beta) \|(u,v)\|_{\mathbb{D}}^{2} + (\alpha + \beta -2_{s_{1}} )  \|u\|_{{2_{s_{1}}^{}}}^{2_{s_{1}}^{}}+ (\alpha + \beta -2_{s_{2}})\|v\|_{{2_{s_{2}}^{}}}^{2_{s_{2}}^{}} < 0 . \numberthis \label{second order derivative}
 \end{align*} 
This implies that 
\begin{align} \label{max}
	J_{\kappa}(\eta_{(u,v)}u,\eta_{(u,v)}v) = \max_{\zeta > 0}J_{\kappa}(\zeta u,\zeta v) > J_{\kappa}(\zeta u,\zeta v) \text{ for } 0 < \zeta \neq \eta_{(u,v)}.
\end{align}
Moreover, by the identity \eqref{equivalent norm} there exists a $r_{\kappa} > 0$ such that the product norm is bounded from below i.e.,
\begin{align} \label{ norm equal  r}
    \|(u,v)\|_{\mathbb{D}} > r_{\kappa} ~~\mbox{for~all}~(u,v) \in \mathcal{N}_{\kappa}.
\end{align}
Now using the Lagrange multiplier method, if $(u, v) \in \mathbb{D}$ is a critical point of $J_{\kappa}$ on the Nehari manifold $ \mathcal{N}_{\kappa}$,  then there exists a $\rho \in \mathbb{R}$ called Lagrange multiplier such that
\[ (J_{\kappa}|_{\mathcal{N}_{\kappa} } )'(u,v) =  J_{\kappa}'(u,v) - \rho \Phi'_{\kappa}(u,v) = 0.\]
Further calculations yields $\rho \langle \Phi'_{\kappa}(u,v)|(u,v) \rangle = \langle J'_{\kappa}(u,v)|(u,v) \rangle =0 $. It is clear that $\rho=0$, otherwise the inequality (\ref{second order derivative}) does not hold and subsequently we obtain $J_{\kappa}'(u,v) = 0$. Hence, there is a one-to-one correspondence between the critical points of $J_{\kappa}$ and the critical points of $J_{\kappa}|_{\mathcal{N}_{\kappa}}$. The functional $J_{\kappa}$ restricted on ${\mathcal{N}_{\kappa}}$ can also be written as
\begin{equation} \label{two one four}
        ( J_{\kappa}|_{\mathcal{N}_{\kappa}})(u,v) = \bigg( \frac{1}{2} - \frac{1}{\alpha+\beta}  \bigg) \|(u,v)\|_{\mathbb{D}}^{2} + \(\frac{1}{\alpha+\beta} -\frac{1}{2_{s_{1}}} \) \|u\|_{{2_{s_{1}}}}^{2_{s_{1}}} + \(\frac{1}{\alpha+\beta} -\frac{1}{2_{s_{2}}} \) \|v\|_{{2_{s_{2}}^{}}}^{2_{s_{2}}}.
\end{equation} 
Now the hypotheses (\ref{ alpha beta condition}) and (\ref{ norm equal  r}) combined with \eqref{two one four} yields the following inequality
\[J_{\kappa}(u,v) > \bigg( \frac{1}{2} - \frac{1}{\alpha+\beta}  \bigg) r_{\kappa}^{2} ~~\mbox{for~all}~(u,v) \in \mathcal{N}_{\kappa}.\]
Notice that the functional $J_\kappa$ restricted on $\mathcal{N}_{\kappa}$ is bounded from below. Hence, we can expect the solution of \eqref{main problem} by minimizing the energy functional $J_\kappa$ on the Nehari manifold $\mathcal{N}_\kappa$. Further, such a minimizer on the Nehari manifold $\mathcal{N}_\kappa$ is a ground state solution since $\mathcal{N}_\kappa$ consists of all the non-trivial critical points of the functional $J_\kappa$.

Next we prove that the (PS) sequences on $\mathcal{N}_{\kappa}$ are bounded in the product space $\mathbb{D}$.
\begin{lemma} \label{equivalent of critical points lemma}
	If $\{(u_{n},v_{n})\} \subset \mathcal{N}_{\kappa}$ is a (PS) sequence for $J_{\kappa}$ on ${\mathcal{N}_{\kappa}}$ at level $c \in \mathbb{R}$ under the hypotheses (\ref{ alpha beta condition}) and (\ref{condition on h}), then $\{(u_{n},v_{n})\}$ is a bounded (PS) sequence for $J_{\kappa}$ in $\mathbb{D}$.
\end{lemma}
\begin{proof}
	Given $\{(u_{n},v_{n})\}\subset \mathcal{N}_{\kappa} $ is a (PS) sequence for $J_{\kappa}$ at level $c$, then by definition
	\begin{align}
		\begin{split}
			J(u_{n},v_{n}) \rightarrow c~\text{in }\R~\text{as}~~n \rightarrow \infty ,~\text{i.e.,}~ c+ o_n(1) = J(u_{n},v_{n}). \hspace{1cm}
		\end{split}
	\end{align}
	The functional $J_\kappa$ at $(u_n,v_n)$ is given as
	\begin{align}
		J(u_{n},v_{n}) = \frac{1}{2} \|(u_{n},v_{n})\|_{\mathbb{D}}^{2} - \frac{1}{2_{s_{1}}}\|u_{n}\|_{{2_{s_{1}}}}^{2_{s_{1}}} - \frac{1}{2_{s_{2}}}\|v_{n}\|_{{2_{s_{2}}}}^{2_{s_{2}}} - \kappa \int_{\R^2} h(x,y)|u_{n}|^{\alpha}|v_{n}|^{\beta}\dxy.
	\end{align}
	Using the identity \eqref{equivalent norm}, we have
	\begin{align*}
		J(u_{n},v_{n}) &= \frac{1}{2} \|(u_{n},v_{n})\|_{\mathbb{D}}^{2} - \frac{1}{2_{s_{1}}}\|u_{n}\|_{{2_{s_{1}}}}^{2_{s_{1}}} - \frac{1}{2_{s_{2}}}\|v_{n}\|_{{2_{s_{2}}}}^{2_{s_{2}}}
		-\frac{1}{\alpha+\beta} \bigg( \|(u_{n},v_{n})\|_{\mathbb{D}}^{2}-\|u_{n}\|_{{2_{s_{1}}}}^{2_{s_{1}}} -\|v_{n}\|_{{2_{s_{2}}}}^{2_{s_{2}}} \bigg) \\
		&= \bigg(\frac{1}{2}  -\frac{1}{\alpha+\beta}\bigg) \|(u_{n},v_{n})\|_{\mathbb{D}}^{2}+ \bigg(  \frac{1}{\alpha+\beta}-\frac{1}{2_{s_{1}}}\bigg) \|u_{n}\|_{{2_{s_{1}}}}^{2_{s_{1}}} + \bigg(  \frac{1}{\alpha+\beta}-\frac{1}{2_{s_{2}}}\bigg)\|v_{n}\|_{{2_{s_{2}}}}^{2_{s_{2}}}\\
		c + o_n(1) & \geq \bigg(\frac{1}{2}  -\frac{1}{\alpha+\beta}\bigg) \|(u_{n},v_{n})\|_{\mathbb{D}}^{2}.
	\end{align*}
	Thus, the sequence $\{ (u_{n},v_{n})\}$ is bounded in $\mathbb{D}$. Next, we combine (\ref{phi function}) with (\ref{second order derivative}) and (\ref{ norm equal  r}) to obtain  
	\begin{align} \label{phi r inequality}
		\langle\Phi'_{\kappa}(u_{n},v_{n})|(u_{n},v_{n})\rangle \leq (2-\alpha-\beta) r_{\kappa}^{2}.
	\end{align}
	By the Lagrange multiplier method, we can assume the sequence of multipliers $\{ \omega_{n}\} \subset \mathbb{R}$ such that
	\begin{align} \label{J restricted on N with LM}
		(J_{\kappa}|_{\mathcal{N}_{\kappa} } )'(u_{n},v_{n}) =  J_{\kappa}'(u_{n},v_{n}) - \omega_{n} \Phi'_{\kappa}(u_{n},v_{n}) ~\text{in the dual space}~\mathbb{D}^{*}.
	\end{align}
	We know that $(J_{\kappa}|_{\mathcal{N}_{\kappa} } )'(u_{n},v_{n})$ converges to $0$ as $n \rightarrow \infty$ in the dual space $\mathbb{D}^{*}$, which implies $$\langle (J_{\kappa}|_{\mathcal{N}_{\kappa} } )'  (u_{n},v_{n}) |  (u_{n},v_{n}) \rangle \to 0 \text{ as } n \rightarrow \infty. $$ Since $\langle \Phi'_{\kappa}(u_{n},v_{n}) | (u_{n},v_{n}) \rangle <0$, we have $\omega_{n} \rightarrow 0 ~\text{in}~\mathbb{R}~\text{as}~n \rightarrow \infty$. Hence, the result is proved by \eqref{J restricted on N with LM}.
\end{proof}
In the following, we prove that the (PS) sequence in $\mathbb{D}$ is also bounded in $\mathbb{D}$.
\begin{lemma} \label{boundedness lemma}
	If $\{(u_{n},v_{n})\} \subset \mathbb{D}$ is a (PS) sequence for $J_{\kappa}$ at level $c \in \mathbb{R}$ under the hypotheses (\ref{ alpha beta condition}) and (\ref{condition on h}), then the sequence $\{(u_{n},v_{n})\}$ is bounded in $\mathbb{D}$.
\end{lemma}
\begin{proof}
	Given $\{(u_{n},v_{n})\}\subset \mathbb{D}$ is a (PS) sequence for $J_{\kappa}$ at level $c$, then as $n \rightarrow \infty$
	\begin{align}
		J_{\kappa}(u_{n},v_{n}) &\rightarrow c~\text{in}~\mathbb{R} ,\label{e1}\\
		J'_{\kappa}(u_{n},v_{n}) &\rightarrow 0 ~\text{in} ~\mathbb{D}^*. \label{e2}
	\end{align}
	Using \eqref{e2}, we can write
	\begin{align*}
		\bigg< J'_{\kappa}(u_{n},v_{n}) \bigg| \frac{(u_{n},v_{n})}{\|(u_{n},v_{n})\|_{\mathbb{D}}} \bigg> \rightarrow 0 ~\text{in}~\mathbb{R}.
	\end{align*}
	Therefore, we have the following expression
	\begin{align*}
		\|(u_{n},v_{n})\|_{\mathbb{D}}^{2} -\|u_{n}\|_{{2_{s_{1}}}}^{2_{s_{1}}} - \|v_{n}\|_{{2_{s_{2}}}}^{2_{s_{2}}} - \kappa (\alpha + \beta) \int_{\R^2} h(x,y)|u_{n}|^{\alpha}|v_{n}|^{\beta}\dxy = o_n(\|(u_{n},v_{n})\|_{\mathbb{D}})~\text{as}~n \rightarrow \infty,
	\end{align*}
	and using \eqref{e1} we get the following 
	\begin{align*}
		\frac{1}{2} \|(u_{n},v_{n})\|_{\mathbb{D}}^{2} - \frac{1}{2_{s_{1}}}\|u_{n}\|_{{2_{s_{1}}}}^{2_{s_{1}}} - \frac{1}{2_{s_{2}}}\|v_{n}\|_{{2_{s_{2}}}}^{2_{s_{2}}} - \kappa \int_{\R^2} h(x,y)|u_{n}|^{\alpha}|v_{n}|^{\beta}\dxy = c+o_n(1) ~\text{as}~n \rightarrow \infty.
	\end{align*}
	From \eqref{e1} and \eqref{e2}, we can write
	\begin{align*}
		J_{\kappa}(u_{n},v_{n}) - \frac{1}{\alpha+\beta} \langle J'_{\kappa}(u_{n},v_{n}) |  \frac{(u_{n},v_{n})}{\|(u_{n},v_{n})\|_{\mathbb{D}}} \rangle = c + o_n(1)+ o_n(\|(u_{n},v_{n})\|_{\mathbb{D}}) ~\text{as}~n \rightarrow \infty,
	\end{align*}
	which further gives the following inequality
	\begin{align*}
		\bigg(\frac{1}{2}  -\frac{1}{\alpha+\beta} \bigg) \|(u_{n},v_{n})\|_{\mathbb{D}}^{2} &\leq \bigg(\frac{1}{2}  -\frac{1}{\alpha+\beta} \bigg) \|(u_{n},v_{n})\|_{\mathbb{D}}^{2} +\bigg(\frac{1}{\alpha+\beta} -\frac{1}{{2_{s_{1}}}} \bigg) \|u_{n}\|_{{2_{s_{1}}}}^{2_{s_{1}}} \\
		&\qquad + \bigg(\frac{1}{\alpha+\beta} -\frac{1}{{2_{s_{2}}}} \bigg)\|v_{n}\|_{{2_{s_{2}}}}^{2_{s_{2}}}\\
		&= c + o_n(1)+ o_n(\|(u_{n},v_{n})\|_{\mathbb{D}}) ~\text{as}~ n\rightarrow \infty.
	\end{align*}
	Hence, the sequence $\{(u_{n},v_{n})\}$ is bounded in $\mathbb{D}$.
\end{proof}

\subsection{Fractional Laplacian operator}
	
We denote $(-\Delta)_y$ as the classical Laplacian operator w.r.t $y$ variable and its fractional analogue i.e., the fractional Laplacian operator w.r.t $y$ variable denoted by $(-\Delta)^s_y, s\in (0,1)$ is defined on smooth functions as
\begin{align*}
	(-\Delta)^s_y u(x,y) =C(N,s)\  P.V. \int_{\mathbb{R}^N} \frac{u(x,y)-u(x,z)}{|y-z|^{N+2s}} \mathrm{d}z, \text{ for all } (x,y)  \in \mathbb{R}^N \times \mathbb{R}^M,
\end{align*} where $C(N,s):= \big(\int_{\mathbb{R}^N} \frac{1- cos(\zeta_1)}{|\zeta |^{N+2s}}\, \d \zeta \big)^{-1}$ is a normalizing constant and the abbreviation P.V. mean{\color{blue}s} the principal value sense. Since we are dealing our problem \eqref{main problem} in $\mathbb{R}^2$, we choose $N=1$ and $M=1$. So the formal definition takes the following form
\begin{align} \label{defintion fractional laplacian}
	(-\Delta)^s_y u(x,y) =C(s)\ P.V. \int_{\mathbb{R}} \frac{u(x,y)-u(x,z)}{|y-z|^{1+2s}} \mathrm{d}z, \text{ for all } (x,y)  \in \mathbb{R}^2,
\end{align}
where 
\begin{equation}\label{Cs}
	C(s):= \bigg(\int_{\mathbb{R}} \frac{1- cos(\zeta)}{|\zeta |^{1+2s}} \, \d \zeta \bigg)^{-1}.
\end{equation}
Due to the singularity of the kernal, the right-hand side of \eqref{defintion fractional laplacian} is not well defined in general. In fact for $s \in (0, 1/2)$ the integral in \eqref{defintion fractional laplacian} is not really singular near $y$. In the following, we prove the analogous results of Hitchiker's guide \cite{NPV2012}.

\begin{lemma} \label{l3}
	Let $s\in (0,1)$ and let $(- \Delta)_y^s$ be the fractional Laplacian operator defined by \eqref{defintion fractional laplacian}. Then, for any $u \in \mathscr{P}(\mathbb{R}^2)$ (Schwartz space of rapidly decaying functions),
	\begin{equation} \label{second order}
		(- \Delta)_y^s u(x,y) = - \frac{1}{2} C(s) \int_{\mathbb{R}} \frac{u(x,y+z) + u(x, y-z)- 2u(x,y)}{|z|^{1+2s}}\ \mathrm{d}z, \ \forall (x,y) \in \mathbb{R}^2.
	\end{equation}
\end{lemma}
\begin{proof}
	By choosing $t= z-y$, we have
	\begin{align} \label{a1}
		(- \Delta)_y^s u(x,y) &= -C(s)\ P.V. \int_{\mathbb{R}} \frac{u(x,z)-u(x,y)}{|y-z|^{1+2s}} \mathrm{d}z = -C(s)\ P.V. \int_{\mathbb{R}} \frac{u(x,y+t)-u(x,y)}{|t|^{1+2s}} \mathrm{d}t.
	\end{align} 
	Now by substituting $\Tilde{t}=-t$ in last term of the above equality, we have
	\begin{align}
		P.V. \int_{\mathbb{R}} \frac{u(x,y+t)-u(x,y)}{|t|^{1+2s}} \mathrm{d}t = P.V. \int_{\mathbb{R}} \frac{u(x,y-\Tilde{t})-u(x,y)}{|\Tilde{t}|^{1+2s}} \mathrm{d}\Tilde{t}
	\end{align}
	So after relabeling $\Tilde{t}$ as $t$
	\begin{align}\label{a2}
		\begin{split}
			2P.V. \int_{\mathbb{R}} \frac{u(x,y+t)-u(x,y)}{|t|^{1+2s}} \mathrm{d}t &= P.V. \int_{\mathbb{R}} \frac{u(x,y+t)-u(x,y)}{|t|^{1+2s}} \mathrm{d}t + P.V. \int_{\mathbb{R}} \frac{u(x,y-t)-u(x,y)}{|t|^{1+2s}} \mathrm{d}t\\
			&= P.V. \int_{\mathbb{R}} \frac{u(x,y+t)+ u(x,y-t)-2u(x,y)}{|t|^{1+2s}} \mathrm{d}t.
		\end{split}
	\end{align}
	Therefore, if we replace $t$ by $z$ in \eqref{a1} and \eqref{a2}, we can write the fractional Laplacian operator in \eqref{defintion fractional laplacian} as
	\[ (- \Delta)_y^s u(x,y) = - \frac{1}{2} C(s)\  P.V.\int_{\mathbb{R}} \frac{u(x,y+z) + u(x, y-z)- 2u(x,y)}{|z|^{1+2s}}\ \mathrm{d}z.\]
	The above representation is useful to remove the singularity at the origin. Indeed, for any smooth function $u$, a second order Taylor expansion yields,
	\[\frac{u(x,y+z) + u(x, y-z)- 2u(x,y)}{|z|^{1+2s}}\leq \frac{\|D^2u\|_{L^\infty}}{|z|^{2s-1}} \]
	which is integrable near $0$ (for any fixed $s \in (0,1)$). Therefore, since $u \in \mathscr{P}(\R^2)$, one can get rid of the $P.V.$ and write \eqref{second order}.
\end{proof}
\begin{remark}
    It can be noticed that the operator $(-\Delta)^s_y$ is indeed well defined for any $u \in C^2(\R^2)\cap L^\infty(\R^2)$. If we take any non-zero constant, it does not belong to the Schwartz space $\mathscr{P}(\R^2)$ but it is fractional harmonic.
\end{remark}
\noindent For any $u \in \mathscr{P}(\mathbb{R}^2)$, the Fourier transform of $u$ is well-defined and it is given by
\[
\mathscr{F}u(\xi) = \hat{u}(\xi) = \frac{1}{2 \pi} \int_{\mathbb{R}^2} e^{-i \xi \cdot x} u(x) dx, \ \text{ for } \xi = (\xi_1, \xi_2) \in \R^2.
\]
\begin{proposition}
	Let $s \in (0,1)$ and let $(-\Delta)^s_y : \mathscr{P}(\R^2) \rightarrow L^2(\mathbb{R}^2)$ be the fractional Laplacian operator defined by \eqref{defintion fractional laplacian}. Then, for any $u \in \mathscr{P}(\R^2)$,
	\begin{equation} \label{laplacian in Fourier term}
		(-\Delta)_y^s u = \mathscr{F}^{-1} \big( |\xi_2|^{2s} (\mathscr{F}u)\big), \ \forall \xi=(\xi_1, \xi_2) \in \mathbb{R}^2.
	\end{equation}
\end{proposition} \label{P1}
\begin{proof}
	The Lemma \ref{l3} represents the fractional Laplacian w.r.t. $y$ variable as a second order differential quotient given by \eqref{second order}. We denote by $\mathscr{K}u$ the integral in \eqref{second order}, i.e.
	\[\mathscr{K}u(x,y) = - \frac{1}{2} C(s) \int_{\mathbb{R}} \frac{u(x,y+z) + u(x, y-z)- 2u(x,y)}{|z|^{1+2s}}\ \mathrm{d}z.\]
	$\mathscr{K}$ is a linear operator and our goal is to look for a function $A: \mathbb{R} \rightarrow \mathbb{R}$ such that
	\begin{equation}
		\mathscr{K}u = \mathscr{F}^{-1}\big( (A(\mathscr{F}u))\big).
	\end{equation}
	Indeed, we seek for $A$ satisfying 
	\begin{equation} \label{Symbol A}
		A(\xi_2) = |\xi_2|^{2s}.
	\end{equation}
	Further, we have the following estimate
	\begin{align*}
		\frac{|u(x,y+z)+ u(x,y-z)-2u(x,y)|}{|z|^{1+2s}} \leq &4 \big( \chi_{B_1} |z|^{1-2s} \sup_{B_1}|D^2u| \\
		& + \chi_{\mathbb{R} \setminus B_1} |z|^{-1-2s}|u(x,y+z)+ u(x,y-z)-2u(x,y)| \big) \in L^{1}(\mathbb{R}^2).
	\end{align*}
	Hence, by the Fubini-Tonelli theorem we can exchange the integral in $z$ with the Fourier transform in $(x,y)$. Applying Fourier transform we obtain
	\begin{align*}
		A(\xi_2) (\mathscr{F}u)(\xi) &= \mathscr{F}(\mathscr{K}u)\\
		&= - \frac{1}{2} C(s) \int_{\mathbb{R}} \frac{\mathscr{F}(u(x,y+z) + u(x, y-z)- 2u(x,y))}{|z|^{1+2s}}\ \mathrm{d}z\\
		&= - \frac{1}{2} C(s) \int_{\mathbb{R}} \frac{e^{i \xi_2 z} + e^{-i \xi_2 z}-2}{|z|^{1+2s}} \mathrm{d}z (\mathscr{F}u)(\xi)\\
		&= C(s) \int_{\mathbb{R}} \frac{1-cos(\xi_2 z)}{|z|^{1+2s}} \mathrm{d}z (\mathscr{F}u)(\xi).
	\end{align*}
	For claiming \eqref{Symbol A} it is enough to prove that 
	\begin{equation} \label{H e1}
		C(s) \int_{\mathbb{R}} \frac{1-cos(\xi_2 z)}{|z|^{1+2s}} \mathrm{d}z = |\xi_2|^{2s}.
	\end{equation}
	Now, we consider a function $\mathcal{I}: \mathbb{R} \rightarrow \mathbb{R}$ defined by
	\[\mathcal{I}(\xi_2) = \int_{\mathbb{R}} \frac{1-cos(\xi_2 z)}{|z|^{1+2s}} \mathrm{d}z. \]
	We observe that $\mathcal{I}(-\xi_2) = \mathcal{I}(\xi_2)$ and therefore we have
	\begin{align*}
		\mathcal{I}(\xi_2) &= \int_{\mathbb{R}} \frac{1-cos(\xi_2 z)}{|z|^{1+2s}} \mathrm{d}z\\
		&= |\xi_2|^{2s} \int_{\mathbb{R}} \frac{1-cos(t)}{|t|^{1+2s}} \mathrm{d}t = |\xi_2|^{2s} C(s)^{-1},
	\end{align*}
	which proves \eqref{H e1} and hence the proof is complete.
\end{proof}
\begin{proposition}
	Let $s\in (0,1)$. Then for any $u \in H^s(\R^2)$ we have
	\[ 2 C(s)^{-1} \int_{\R^2} |\xi_2|^{2s} |\mathscr{F}u (\xi)|^2 \d \xi= \int_{\R^3} \frac{|u(x,y)-u(x,z)|^2}{|y-z|^{1+2s}} \d z \dxy,\]
	where $C(s)$ is given by \eqref{Cs} and $ H^s(\R^2)$ is the fractional Sobolev space.
\end{proposition}
\begin{proof}
	First we compute
	\begin{align*}
		\mathscr{F} \bigg( \frac{u(x,z+ y) -u(x, y)}{|z|^{\frac{1}{2} +s}}\bigg) &= \frac{1}{2\pi} \frac{1}{|z|^{\frac{1}{2}+s}} \int_{\mathbb{R}^2} e^{-i(\xi_1,\xi_2) \cdot (x,y)}\big(u(x,z+y)-u(x,y) \big) \dxy\\
		&=  \frac{(e^{i \xi_2 z}-1)}{|z|^{\frac{1}{2}+s}} \mathscr{F}u. \numberthis
	\end{align*} 
	Now we calculate
	\begin{align*}
		\int_{\R}\bigg\|\mathscr{F} \( \frac{u(x,z+y)-u(x,y)}{|z|^{\frac{1}{2}+s}} \) \bigg\|_{L^2(\R^2)}^2 \d z &= \int_{\R} \int_{\R^2} \frac{|e^{i \xi_2 z}-1|^2}{|z|^{1+2s}} |\mathscr{F}u (\xi_1, \xi_2)|^2 \d \xi_1 \d \xi_2 \d z\\
		&= 2 \int_{\R} \int_{\R^2} \frac{1-cos(\xi_2 z)}{|z|^{1+2s}} |\mathscr{F}u (\xi)|^2 \d \xi \d z\\
		&= 2 C(s)^{-1} \int_{\R^2} |\xi_2|^{2s} |\mathscr{F}u (\xi)|^2 \d \xi. \numberthis \label{F norm}
	\end{align*}
	On the other hand, by the Plancherel formula we get
	\begin{align*}
		\int_{\R}\bigg\|\mathscr{F} \( \frac{u(x,z+y)-u(x,y)}{|z|^{1/2+s}} \) \bigg\|_{L^2(\R^2)}^2 \d z&= \int_{\R}\bigg\|  \frac{u(x,z+y)-u(x,y)}{|z|^{1/2+s}} \bigg\|_{L^2(\R^2)}^2 \d z\\
		&= \int_{\R} \( \int_{\R^2}  \frac{|u(x,z+y)-u(x,y)|^2}{|z|^{1+2s}} \dxy \)\d z\\
		&= \int_{\R^2} \( \int_{\R}  \frac{|u(x,z+y)-u(x,y)|^2}{|z|^{1+2s}}\d z  \)\dxy \\
		&= \int_{\R^3} \frac{|u(x,y)-u(x,z)|^2}{|y-z|^{1+2s}} \d z \dxy.  \numberthis \label{G norm}
	\end{align*}
	By \eqref{F norm} and \eqref{G norm}, we obtain
	\begin{align*}
		2 C(s)^{-1} \int_{\R^2} |\xi_2|^{2s} |\mathscr{F}u (\xi)|^2 \d \xi= \int_{\R^3} \frac{|u(x,y)-u(x,z)|^2}{|y-z|^{1+2s}} \d z \dxy.
	\end{align*}
\end{proof}
\begin{remark}
	Notice that the above Proposition is important in the sense that it gives us an equivalence relation between the Fourier norm and the Gagliardo type semi-norm. Since our problem is of variational type, it is natural to work with the Gagliardo type semi-norm than Fourier norm. Moreover, we can write (see also \cite[Remark 1.1]{Esfahani2018})
	\begin{equation}
		\|(-\Delta)_y^{s/2}u\|_{L^{2}(\R^2)}^2 = \frac{C(s)}{2} \int_{\R^3} \frac{|u(x,y)-u(x,z)|^2}{|y-z|^{1+2s}} \d z \dxy.
	\end{equation}
\end{remark}
We define the fractional gradient w.r.t $y$ variable for $u \in H^s(\R^2)$ as 
\begin{equation} \label{Fractional Gradient}
    |D^{s}_{y}u(x,y)|^2= \frac{C(s)}{2} \int_{\R} \frac{|u(x,y)-u(x,z)|^2}{|y-z|^{1+2s}} \d z.
\end{equation} 
The next two lemma are due to \cite[Lemma 2.2, Lemma 2.4]{Bonder2018}. The proof follows by similar arguments and therefore we omit the proof.
\begin{lemma} \label{Bonder 1}
	Let $v \in W^{1,\infty}(\R^2)$ be such that $supp(v) \subset B_1(0)$. Then, there exits a constant $C>0$ depending on $s$ and $\|v\|_{1,\infty}$ such that
	\begin{equation}
		|D^s_y v(x,y)| \leq C \min\{ 1, (x^2 + y^2)^{- \frac{1+2s}{2}} \}.
	\end{equation}
\end{lemma}
\begin{lemma} \label{Bonder 2}
	Let $0<s<1$ and $w\in L^{\infty}(\R^2)$ be such that there exist $\alpha>0$ and $C>0$ such that \[0 \leq w(x,y) \leq C (x^2 + y^2)^{-\frac{\alpha}{2}}.\]
	If $\alpha> 2s$, then $\H \subset \subset L^2(w \,\dxy;\R^2)$, where $L^2(w \,\dxy;\R^2)$ is the weighted $L^2$-space with weight $w$ and the symbol $``\subset \subset"$ means the compact embedding.
\end{lemma}
\begin{remark} \label{Remark w}
	If $\phi \in W^{1,\infty}(\R^2)$ has compact support, then $w = |D^s_y \phi|^2$ verifies all the assumptions of the above lemma.
\end{remark}

\section{Concentration- Compactness}\label{concentration sec}

In this section, we recall the following result by Lions \cite[Lemma 1.2]{Lionslimitcase}, which we will use further to prove the concentration-compactness results.
\begin{lemma} \label{Lions lemma}
   Let $\nu,\mu$ be two non-negative, bounded measures on $\mathbb{R}^{N}$ such that
   \[ \bigg[ \int_{\mathbb{R}^N} |\phi|^q\, \mathrm{d}\nu \bigg]^{\frac{1}{q}} \leq C \bigg[ \int_{\mathbb{R}^N} |\phi|^p\, \mathrm{d}\mu \bigg]^{\frac{1}{p}},~~ \forall ~\phi \in C_{c}^{\infty}(\mathbb{R}^N), \]
   for some constant $C>0$ and $1\leq p<q<\infty$. Then there exist a countable set $\{x_k \in \mathbb{R}^N : k \in \mathbb{K}\}$ and $\nu_k \in (0, \infty)$ such that 
   \[ \nu = \sum_{k \in \mathbb{K}} \nu_k \delta_{x_k}. \]
\end{lemma}
\begin{lemma}[\cite{Brezis_Lieb_1983}, Brezis-Lieb Theorem 1] \label{Brezis Lieb lemma}
    Let $r \geq 1$ and $\{u_n\}_{n \in \mathbb{N}} \subset L^{r}(\R^N)$ be a bounded sequence such that $u_n \rightarrow u$ a.e. in $\R^N$, then 
    \[\|u_n\|^r_{L^{r}(\R^N)} - \|u_n -u\|^r_{L^{r}(\R^N)} = \|u\|^r_{L^{r}(\R^N)} + o_n(1). \]
\end{lemma}
\begin{lemma}\cite[Theorem 1.1]{Esfahani2015}
    Let $s \in (0,1)$ and $2 \leq q \leq 2_s=\frac{2(1+s)}{1-s}$. For any $u \in \mathcal{H}^{1,s}(\mathbb{R}^{2})$, there exists a constant $C_{q,s}>0$ such that
    \begin{equation} \label{inequality 1}
        \int_{\R^2}|u|^q \dxy \leq C_{q,s}  \(\int_{\R^2}|u|^2 \dxy\)^{\frac{q}{2}- \frac{(q-2)(s+1)}{4s}} \(\int_{\R^2}|\partial_x u|^2 \dxy\)^{\frac{q-2}{4}} \(\int_{\R^2}|(-\Delta)_y^{s/2}u|^2 \dxy\)^{\frac{q-2}{4s}}.
    \end{equation}
\end{lemma}
Note that all the exponents on the right hand side of \eqref{inequality 1} are positive for $2<q<2_s$. Raising \eqref{inequality 1} to power $\frac{2}{q}$, we obtain that
\begin{equation*} 
        \(\int_{\R^2}|u|^q \dxy\)^{2/q} \leq C_{q,s}  \(\int_{\R^2}|u|^2 \dxy\)^{1- \frac{2(q-2)(s+1)}{4sq}} \(\int_{\R^2}|\partial_x u|^2 \dxy\)^{\frac{2(q-2)}{4q}} \(\int_{\R^2}|(-\Delta)_y^{s/2}u|^2 \dxy\)^{\frac{2(q-2)}{4sq}}.
\end{equation*}
As a consequence, for all $q: 2<q<2_s$, we can apply the Young's inequality to obtain the following
\begin{equation} \label{Sobolev type embedding}
     \(\int_{\R^2}|u|^q \dxy\)^{2/q} \leq C \int_{\R^2} \big(|u|^2 + |\partial_x u|^2 + |(-\Delta)^{s/2}_yu|^2 \big) \dxy.
\end{equation}
Note that \eqref{Sobolev type embedding} is equivalent to the Sobolev type embedding $\mathcal{H}^{1,s}(\R^2) \hookrightarrow L^q(\R^2)$, which holds under the assumption $s>0$ and $2<q\leq 2_s$ (see also \cite[Remark 2.2]{Esfahani2015}). Let us define
\begin{equation} \label{Best constant}
    \Lambda_{i} := \inf\limits_{u \in \Hi \setminus \{0\}}\frac{\ui}{\u2si}, \ i=1,2.
\end{equation}
By definition it follows that
\begin{equation} \label{Sobolev type 2}
    \Lambda_{i} \|u\|_{2_{s_i}}^2 \leq \|u\|_{\mathcal{H}^{1,s_i}(\R^2)}^2, \ i=1,2,
\end{equation}
where $\Lambda_i >0 ~(i=1,2)$ by embedding \eqref{Sobolev type embedding}.
For the sake of simplicity, let us first deal with the case $s_i=s$ and $\Lambda_i= \Lambda$ for $i=1,2$. It is given that for $u \in \mathcal{H}^{1,s}(\R^2)$,
\[\Dely u(x,y) = C(s)\ P.V. \int_{\R} \frac{u(x,y)-u(x,z)}{|y-z|^{1+s}}\ \d z.\]
Observe that $\|D^{s}_{y}u\|^{2}_{L^{2}(\R^2)} = \|\Dely u\|^{2}_{L^{2}(\R^2)}= \frac{C(s)}{2} \int_{\R^3} \frac{|u(x,y)-u(x,z)|^2}{|y-z|^{1+2s}} \d z \dxy$. Let $\phi \in \C_c$, then by critical Sobolev type embedding \eqref{Sobolev type 2}, we obtain
\begin{equation}
    \Lambda^{1/2} \|\phi u_k\|_{2_s} \leq \|\phi u_k\|_{\H},
\end{equation}
where the sequence $\{u_k\}$ converges weakly to $u$ in $\H$. First we assume that $u=0$, then
\begin{equation*}
    \|\phi u_k\|_{2_s} = \(\int_{\R^2}|\phi u_k|^{2_s}\, \dxy \)^{\frac{1}{2_s}} = \(\int_{\R^2} |\phi|^{2_s} |u_k|^{2_s} \,\dxy \)^{\frac{1}{2_s}}.
\end{equation*}
 Let $\nu_k :=|u_k|^{2_s} \,\dxy \overset{\ast}{\rightharpoonup} \nu$ and $ \mu_k :=\big(|u_k|^2 + |\partial_x u_k|^2 + |D^s_y u_k|^2 \big) \,\dxy \overset{\ast}{\rightharpoonup} \mu$ in the sense of measure. Then
\begin{equation} \label{convergence nu}
    \|\phi u_k\|_{2_s} = \(\int_{\R^2} |\phi|^{2_s} |u_k|^{2_s}\, \dxy \)^{\frac{1}{2_s}} \rightarrow \(\int_{\R^2} |\phi|^{2_s}\, \d \nu \)^{\frac{1}{2_s}} \text{ as } k \rightarrow \infty.
\end{equation}
Now
\begin{align*}
    \|\phi u_k\|_{\H}^2 &= \int_{\R^2} \big( |\phi u_k|^2 + |\partial_x(\phi u_k)|^2+ |D_y^s(\phi u_k)|^2 \big)\, \dxy\\
    &=\underbrace{\int_{\R^2} |\phi u_k|^2\, \dxy}_\text{I}  + \underbrace{\int_{\R^2} |\partial_x(\phi u_k)|^2 \,\dxy}_\text{II} + \underbrace{\int_{\R^2} |D_y^s(\phi u_k)|^2\, \dxy}_\text{III} .
\end{align*}
\begin{align*}
    \text{III} &= \int_{\R^2} |D^s_y (\phi u_k)|^2\, \dxy =  \int_{\R^2} |\Dely (\phi u_k)|^2\, \dxy\\
    &= \frac{C(s)}{2} \int_{\R^3} \frac{|(\phi u_k)(x,y)-(\phi u_k)(x,z)|^2}{|y-z|^{1+2s}}\, \d z \dxy\\
    &= \frac{C(s)}{2} \int_{\R^3} \frac{|\phi(x,y+h) u_k(x,y+h)- \phi(x,y) u_k(x,y)|^2}{|h|^{1+2s}}\, \d h \dxy\\
    &= \frac{C(s)}{2} \int_{\R^3}\frac{|\phi(x,y+h) u_k(x,y+h)- \phi(x,y)u_k(x,y+h) + \phi(x,y)u_k(x,y+h) - \phi(x,y) u_k(x,y)|^2}{|h|^{1+2s}}\, \d h \dxy\\
    &= \frac{C(s)}{2} \int_{\R^3}\frac{|\big(\phi(x,y+h) - \phi(x,y)\big)u_k(x,y+h) + \phi(x,y) \big(u_k(x,y+h) - u_k(x,y) \big)|^2}{|h|^{1+2s}}\, \d h \dxy.
\end{align*}
Now the application of Minkowski's inequality gives the following estimate
\begin{align*}
    \text{III} &\leq \frac{C(s)}{2} \bigg[ \(\int_{\R^3}\frac{|\phi(x,y+h) - \phi(x,y)|^2|u_k(x,y+h)|^2}{|h|^{1+2s}}\, \d h \dxy\)^{1/2} \\
    &\qquad + \(\int_{\R^3}\frac{|u_k(x,y+h) - u_k(x,y)|^2|\phi(x,y)|^2}{|h|^{1+2s}}\, \d h \dxy\)^{1/2}    \bigg]^2\\
    &= \bigg[ \(\frac{C(s)}{2}\int_{\R^3}\frac{|\phi(x,y+h) - \phi(x,y)|^2|u_k(x,y+h)|^2}{|h|^{1+2s}}\, \d h \dxy\)^{1/2} \\
    &\qquad + \(\frac{C(s)}{2}\int_{\R^3}\frac{|u_k(x,y+h) - u_k(x,y)|^2|\phi(x,y)|^2}{|h|^{1+2s}}\, \d h \dxy\)^{1/2}    \bigg]^2\\
    &= \bigg[ \(\int_{\R^2} |\phi(x,y)|^2 |D^s_y u_k(x,y)|^2\, \dxy\)^{1/2} \\
    &\qquad + \( \underbrace{\frac{C(s)}{2} \int_{\R^3}\frac{|\phi(x,y+h) - \phi(x,y)|^2|u_k(x,y+h)|^2}{|h|^{1+2s}}\, \d h \dxy}_\text{IV}\)^{1/2} \bigg]^2. \numberthis \label{eq1}
\end{align*}
\begin{align*}
    \text{IV} &=\frac{C(s)}{2} \int_{\R^3}\frac{|\phi(x,y+h) - \phi(x,y)|^2|u_k(x,y+h)|^2}{|h|^{1+2s}}\, \d h \dxy\\
    &= \frac{C(s)}{2}\int_{\R^3}\frac{|\phi(x,t) - \phi(x,t-h)|^2|u_k(x,t)|^2}{|h|^{1+2s}}\, \d h \d x \d t\\
    &=\frac{C(s)}{2} \int_{\R^3}\frac{|\phi(x,t) - \phi(x,t+\hat{h})|^2|u_k(x,t)|^2}{|\hat{h}|^{1+2s}}\, \d \hat{h} \d x \d t\\
    &=\frac{C(s)}{2} \int_{\R^3}\frac{|\phi(x,y) - \phi(x,y+h)|^2|u_k(x,y)|^2}{|h|^{1+2s}}\, \d h \d x \d y\\
    &= \int_{\R^2} |u_k(x,y)|^2 |D_y^s \phi (x,y)|^2 \d x \d y. \numberthis \label{eq2}
\end{align*}
From \eqref{eq1} and \eqref{eq2}, we deduce that
\begin{align}
    \text{III}^{1/2} \leq  \(\int_{\R^2} |\phi(x,y)|^2 |D^s_y u_k(x,y)|^2\, \dxy\)^{1/2} + \(\int_{\R^2} |u_k(x,y)|^2 |D_y^s \phi (x,y)|^2 \d x \d y\)^{1/2}.
\end{align}
Using Remark \ref{Remark w}, the weight $w(x,y):= |D^s_y \phi(x,y)|^2$ verifies all the hypothesis of Lemma \ref{Bonder 2}. Therefore, the sequence $\{u_k\}$ converges strongly to $u=0$ in $L^{2}(w\, \dxy; \R^2)$. Hence, we obtain
\begin{equation}
    \lim\limits_{k \rightarrow \infty} \text{III}^{1/2} \leq \lim\limits_{k \rightarrow \infty} \(\int_{\R^2} |\phi(x,y)|^2 |D^s_y u_k(x,y)|^2\, \dxy\)^{1/2}.
\end{equation}
Further,
\begin{equation}
    I = \int_{\R^2} |\phi|^2 |u_k|^2\, \dxy \rightarrow \int_{\R^2} |\phi|^2 |u|^2\, \dxy \text{ as } k \rightarrow \infty
\end{equation}
and 
\begin{align*}
    \text{II}^{1/2} &= \(\int_{\R^2} |\partial_x(\phi u_k)|^2\,  \dxy\)^{1/2} = \(\int_{\R^2} |\phi \partial_x(u_k) + u_k \partial_x(\phi)|^2\,  \dxy\)^{1/2}\\
    &= \|\phi \partial_x(u_k) + u_k \partial_x(\phi)\|_{L^2(\R^2)}\\
    & \leq \|\phi \partial_x(u_k)\|_{L^2(\R^2)} + \|u_k \partial_x(\phi)\|_{L^2(\R^2)}  \text{ (by Minkowski's inequality)}.
\end{align*}
The second term in the last inequality tends to $0$ as $k \rightarrow \infty$ by using a similar set of arguments as in Lemma \ref{Bonder 2}. Subsequently, we have
\begin{equation}
    \lim\limits_{k \rightarrow \infty} \text{II}^{1/2} \leq \lim\limits_{k \rightarrow \infty} \|\phi \partial_x(u_k)\|_{L^2(\R^2)}.
\end{equation}
Thus we have
\begin{align*}
    \Lambda \lim\limits_{k \rightarrow \infty} \|\phi u_k\|_{2_s}^2 \leq \lim\limits_{k \rightarrow \infty} \( \int_{\R^2} |\phi|^2 \big(|u_k|^2 + |\partial_x u_k|^2 + |D^s_y u_k|^2 \big)\, \dxy \).
\end{align*}
Further,
\begin{align*}
   & \Lambda \lim\limits_{k \rightarrow \infty} \(\int_{\R^2}|\phi|^{2_s} |u_k|^{2_s} \dxy\)^{2/2_{s}} \leq \lim\limits_{k \rightarrow \infty} \( \int_{\R^2} |\phi|^2 \big(|u_k|^2 + |\partial_x u_k|^2 + |D^s_y u_k|^2 \big) \,\dxy \)\\
   & \Lambda \lim\limits_{k \rightarrow \infty} \(\int_{\R^2}|\phi|^{2_s}\,  \d \nu_k\)^{2/2_{s}} \leq \lim\limits_{k \rightarrow \infty} \( \int_{\R^2} |\phi|^2\, \d \mu_k \)\\
    & \Lambda \(\int_{\R^2}|\phi|^{2_s}\,  \d \nu\)^{2/2_{s}} \leq  \( \int_{\R^2} |\phi|^2\, \d \mu \).
\end{align*}
\noindent Now by Lemma \ref{Lions lemma}, there exists a countable set $J$, points $\{(x_j,y_j)\}_{j \in J} \subset \R^2$ and positive weights $\{\nu_j\}, \{\mu_j\} \subset \R$ such that 
\begin{align*}
    \nu= \sum_{j\in J} \nu_j \delta_{(x_j,y_j)} + \nu_0 \delta_{(0,0)},  \text{ and } \mu \geq \sum_{j\in J} \mu_j \delta_{(x_j,y_j)} + \mu_0 \delta_{(0,0)},
\end{align*}
where $\delta_{(x_j,y_j)}$ and $\delta_{(0,0)}$ are Dirac functions centred at points $(x_j,y_j)$ and $(0,0)$ respectively.\\
 Now if $u \neq 0$, we replace $u_k$ by $v_k=u_k-u$. Since $u_k \rightharpoonup u$ weakly in $\H$, the sequence $v_k \rightharpoonup 0$ in $\H$. Let $\phi \in C_{b}^{0}(\R^2)$ (the set of all continuous bounded functions), then by Brezis-Lieb lemma \ref{Brezis Lieb lemma} we have
\begin{align*}
   \int_{\R^2} |u_n|^{2_s} \phi\, \dxy  = \int_{\R^2} |u|^{2_s} \phi\, \dxy  + \int_{\R^2} |u_n -u|^{2_s} \phi\, \dxy+ o_n(1).
\end{align*}
Consequently we obtain
\begin{align} \label{concentration 2s}
      |u_n|^{2_s} \overset{\ast}{\rightharpoonup} |u|^{2_s} + \nu = |u|^{2_s} +\sum_{j\in J} \nu_j \delta_{(x_j,y_j)} + \nu_0 \delta_{(0,0)}.
\end{align}
Applying the Brezis Lieb lemma again we have
\begin{align*}
    &\int_{\R^2} \big(|u_n|^2+ |\partial_x (u_n)|^2 + |D^s_y u_n|^2 \big) \phi \, \dxy - \int_{\R^2} \big(|u_n-u|^2+ |\partial_x (u_n-u)|^2 + |D^s_y (u_n-u)|^2 \big) \phi \, \dxy\\
    & = \int_{\R^2} \big(|u|^2+ |\partial_x (u)|^2 + |D^s_y u|^2 \big) \phi \, \dxy + o_n(1),
\end{align*}
which further gives
\begin{align*} \label{concentration e1}
    |u_n|^2+ |\partial_x (u_n)|^2 + |D^s_y u_n|^2 & \overset{\ast}{\rightharpoonup} |u|^2+ |\partial_x (u)|^2 + |D^s_y u|^2 + \mu \\
    & \geq |u|^2+ |\partial_x (u)|^2 + |D^s_y u|^2 + \sum_{j\in J} \mu_j \delta_{(x_j,y_j)} + \mu_0 \delta_{(0,0)}. \numberthis 
\end{align*}
Moreover, we also have
\begin{equation} \label{concentration e2}
    \mu_j \geq \Lambda \nu_j^{2/2_s}, \ \forall j \in J.
\end{equation}

\begin{lemma} \label{PS compactness lemma 1}
    Suppose $\alpha + \beta < \min\{2_{s_1}, 2_{s_2}\}$ or $\alpha + \beta = \min\{2_{s_1}, 2_{s_2}\} < \max\{2_{s_1}, 2_{s_2}\}$ and $h$ satisfies \eqref{condition on h}. Then the functional $J_\kappa$ satisfies the (PS)- condition at level $c$ satisfying 
    \begin{equation} \label{three zero}
        c < \min\bigg\{ \frac{s_1}{1+s_1} \Lambda_1^{\frac{1+s_1}{2s_1}}, \frac{s_2}{1+s_2} \Lambda_2^{\frac{1+s_2}{2s_2}}  \bigg\}.
    \end{equation}
\end{lemma}
\begin{proof}
 Let $\{(u_n,v_n)\}$ be a (PS)-sequence for $J_\kappa$ in $\mathbb{D}$. Then $\{(u_n,v_n)\}$ is a bounded sequence in $\mathbb{D}$ (see Lemma \ref{boundedness lemma}) and by reflexivity there exists a subsequence still denoted by $\{(u_n,v_n)\}$ itself and $(u,v) \in \mathbb{D}$ satisfying the following
\begin{align*}
    (u_{n},v_{n}) &\rightharpoonup (u,v) ~\text{weakly in}~\mathbb{D}, \\
  (u_{n},v_{n}) &\rightarrow (u,v)~\text{strongly in}~ L_{loc}^{q_{1}}(\mathbb{R}^{2}) \times L_{loc}^{q_{2}}(\mathbb{R}^{2})  ~\text{for}~1\leq q_{1} < 2_{s_{1}},1\leq q_{2} < 2_{s_{2}},\\
    (u_{n},v_{n}) &\rightarrow (u,v)  ~\text{a.e.~in}~~\mathbb{R}^{2}.  
\end{align*} 
 Now using \eqref{concentration 2s}, \eqref{concentration e1}   and \eqref{concentration e2}, which are analogous to the concentration–compactness principle of Bonder \cite[Theorem 1.1]{Bonder2018}, there exist a subsequence, still denoted as $\{(u_{n},v_{n})\}$, two at most countable sets of points $\{(x_{j},y_j)\}_{j \in \mathcal{J}} \subset \R^2$ and $\{(\Bar{x}_k, \Bar{y}_{k})\}_{k \in \mathcal{K}} \subset \R^2$, and non-negative numbers $$\{(\mu_{j},\nu_{j})\}_{j \in \mathcal{J}},~ \{(\Bar{\mu}_{k},~\Bar{\nu}_{k})\}_{k \in \mathcal{K}},\mu_{0},~\nu_{0},~\Bar{\mu}_{0} \text{ and }\Bar{\nu}_{0}$$ such that the following convergences hold $weakly^*$ in the sense of measures,
\begin{align} \label{Concentration compactness}
\begin{split}
     |u_n|^2+ |\partial_x u_n|^2 + |D^{s_{1}}_y u_n|^2  \rightharpoonup \d \mu & \geq |u|^2+ |\partial_x u|^2 + |D^{s_1 }_y u|^2 + \sum_{j\in \mathcal{J}} \mu_j \delta_{(x_{j},y_j)} + \mu_0 \delta_{(0,0)},\\
    |v_n|^2+ |\partial_x v_n|^2 + |D^{s_{2}}_y v_n|^2  \rightharpoonup \d \Bar{\mu} & \geq |v|^2+ |\partial_x v|^2 + |D^{s_2}_y v|^2 + \sum_{k\in \mathcal{K}} \Bar{\mu}_k \delta_{(\Bar{x}_k, \Bar{y}_{k})} + \mu_0 \delta_{(0,0)},\\
    |u_{n}|^{2_{s_{1}}}~ \rightharpoonup \mathrm{d}\nu & =|u|^{2_{s_{1}}} + \sum_{j \in \mathcal{J}} \nu_{j}\delta_{(x_{j},y_j)} + \nu_{0}\delta_{(0,0)}, \\
     |v_{n}|^{2_{s_{2}}}~ \rightharpoonup \mathrm{d}\Bar{\nu} &=|v|^{2_{s_{2}}} + \sum_{k \in \mathcal{K}} \Bar{\nu}_{k}\delta_{(\Bar{x}_k, \Bar{y}_{k})} + \Bar{\nu}_{0}\delta_{(0,0)}. 
\end{split}
 \end{align}
Moreover,
 \begin{align} \label{one seven}
     \begin{split}
         \Lambda_{1} \nu_j^{2/2_{s_1}} &\leq \mu_j \text{ for all } j \in \mathcal{J} \cup \{0\},\\
     \Lambda_{2} \Bar{\nu}_k^{2/2_{s_2}} &\leq \Bar{\mu}_k \text{ for all } k \in \mathcal{K} \cup \{0\},
     \end{split}
 \end{align}
where $\delta_{(0,0)},\delta_{(x_{j},y_j)},\delta_{(\Bar{x}_k, \Bar{y}_{k})}$ are the Dirac functions at the points $(0,0),(x_{j},y_j)~\text{and}~(\Bar{x}_k, \Bar{y}_{k})$ of $\R^2$ respectively. We denote the concentration of the sequence $\{u_n\}$ at infinity by the following numbers
\begin{align} \label{Concentration at infty}
    \begin{split}
        \nu_{\infty} &= \lim_{R \rightarrow \infty} \limsup_{n \rightarrow \infty} \int_{\sqrt{x^2 +y^2}>R} |u_{n}|^{2_{s_{1}}}\, \dxy, \\
        \mu_{\infty} &= \lim_{R \rightarrow \infty} \limsup_{n \rightarrow \infty} \int_{\sqrt{x^2 +y^2}>R} \big(|u_n|^2+ |\partial_x u_n|^2 + |D^{s_{1}}_y u_n|^2 \big)\, \dxy.\\
    \end{split}
\end{align}
We similarly define the concentrations of the sequence $\{v_n\}$ at infinity by the numbers $\Bar{\nu}_{\infty}$ and $\Bar{\mu}_{\infty}$. Now let us take a smooth cut-off function centered at $\{(x_j,y_j)\}_{j \in \mathcal{J}} \subset \R^2$ and defined as
\begin{align} \label{test function phi at j}
    \Psi_{j,\epsilon} = 1 ~\text{in}~~B_{\frac{\epsilon}{2}}(x_{j},y_j),~~ \Psi_{j,\epsilon} = 0 ~~\text{in}~~B^{c}_{\epsilon}(x_{j},y_j), ~0\leq \Psi_{j,\epsilon} \leq 1~\text{and}~~|\nabla  \Psi_{j,\epsilon}| \leq \frac{4}{\epsilon},
\end{align}
where $B_r(x_j,y_j) = \{(x,y) \in \R^2 \ : \ (x-x_j)^2 + (y-y_j)^2 < r^2  \}$.
Testing $J'_\kappa(u_n,v_n)$ with $(u_n \Psi_{j, \epsilon},0)$ we have
\begin{align*}
    &\langle J'_\kappa(u_n,v_n) |(u_n \Psi_{j, \epsilon},0) \rangle \\
    &= \int_{\R^2} \big(\partial_x u_n \partial_x(u_n \Psi_{j,\epsilon}) + D^{s_1}_y u_n \cdot D^{s_1}_y (u_n\Psi_{j,\epsilon})+ u_n u_n \Psi_{j,\epsilon} - |u_n|^{2_{s_1}-2}u_n u_n \Psi_{j,\epsilon} \big) \dxy\\
    &\qquad - \kappa \alpha \int_{\R^2} h(x,y) |u_n|^{\alpha-2}u_n u_n  \Psi_{j,\epsilon} |v_n|^{\beta} \dxy\\
    & = \int_{\R^2} \big(\partial_x u_n (u_n\partial_x(\Psi_{j,\epsilon}) + \Psi_{j,\epsilon} \partial_x u_n) + D^{s_1}_y u_n \cdot D^{s_1}_y (u_n\Psi_{j,\epsilon})+ u_n^2 \Psi_{j,\epsilon} - |u_n|^{2_{s_1}} \Psi_{j,\epsilon} \big) \dxy\\
    &\qquad - \kappa \alpha \int_{\R^2} h(x,y) |u_n|^{\alpha}  \Psi_{j,\epsilon} |v_n|^{\beta} \dxy.\\
\end{align*}
Now
\begin{align*}
    \int_{\R^2} D^{s_1}_y u_n \cdot D^{s_1}_y (u_n\Psi_{j,\epsilon})\, \dxy &= \frac{C(s_1)}{2} \int_{\R^3} \frac{|u_n(x,y)-u_n(x,z)|^2}{|y-z|^{1+2s_1}} \Psi_{j, \epsilon}(x,y)\, \d z \dxy\\
    &\quad +\underbrace{ \frac{C(s_1)}{2} \int_{\R^3} \frac{\big(u_n(x,y)-u_n(x,z)\big)  \big(\Psi_{j, \epsilon}(x,y) - \Psi_{j, \epsilon}(x,z)\big)}{|y-z|^{1+2s_1}} u_n(x,z)\, \d z \dxy}_\text{I}.
\end{align*}
\begin{equation}
    \text{I} = \frac{C(s_1)}{2} \int_{\R^3} \frac{\big(u_n(x,y)-u_n(x,z)\big)  \big(\Psi_{j, \epsilon}(x,y)\big) - \Psi_{j, \epsilon}(x,z)}{|y-z|^{1+2s_1}} u_n(x,z)\, \d z \dxy.
\end{equation}
Since $\{u_n \Psi_{j,\epsilon}\}$ is a bounded sequence in $\H$, the Holder's inequality gives
\begin{align*}
    \text{I} &\leq \frac{C(s_1)}{2} \( \int_{\R^3} \frac{|u_n(x,y)-u_n(x,z)|^2}{|y-z|^{1+2s_1}}\, \d z \dxy\)^{1/2} \( \int_{\R^3} \frac{|\Psi_{j, \epsilon}(x,y)-\Psi_{j, \epsilon}(x,z)|^2}{|y-z|^{1+2s_1}}|u_n(x,z)|^2\, \d z \dxy\)^{1/2} \\
    & \leq \frac{C(s_1)}{2} A(s) \int_{\R^3} \frac{|\Psi_{j, \epsilon}(x,y)-\Psi_{j, \epsilon}(x,z)|^2}{|y-z|^{1+2s_1}}|u_n(x,z)|^2\, \d z \dxy\\
    & = A(s) \int_{\R^2} |u_n(x,z)|^2 |D^{s_1}_y \Psi_{j,\epsilon}|^2 \, \d x \d z.
\end{align*}
By using Lemma \ref{Bonder 2}, we have
\begin{align*}
    \lim\limits_{n \rightarrow \infty}\text{I} \leq A(s) \int_{\R^2} |u(x,z)|^2 |D^{s_1}_y \Psi_{j,\epsilon}|^2 \, \d x \d z.
\end{align*}
Observe that $|D^{s_1}_y \Psi_{j, \epsilon}(x,z)|^2 \rightarrow 0$ as $ \epsilon \rightarrow 0$. So taking the limit as $\epsilon \rightarrow 0$ and using the dominated convergence theorem, we can conclude that
\begin{align}
  \lim\limits_{\epsilon \rightarrow 0}  \lim\limits_{n \rightarrow \infty}\text{I} = 0.
\end{align}
Now
\begin{align*}
    \int_{\R^2} \big(\partial_x u_n (u_n\partial_x \Psi_{j,\epsilon}  + \Psi_{j,\epsilon} \partial_x u_n)\,\dxy = \int_{\R^2} \partial_x u_n \partial_x \Psi_{j,\epsilon} u_n\,\dxy + \int_{\R^2} |\partial_x u_n|^2 \Psi_{j,\epsilon}\,\dxy.
\end{align*}
Note that
\[\lim\limits_{\epsilon \rightarrow 0}  \limsup\limits_{n \rightarrow \infty} \int_{\R^2} \partial_x u_n \partial_x (\Psi_{j,\epsilon}) u_n\,\dxy=0.\]
Therefore we have
\[\lim\limits_{\epsilon \rightarrow 0}  \limsup\limits_{n \rightarrow \infty} \int_{\R^2} \big(\partial_x u_n (u_n\partial_x \Psi_{j,\epsilon}  + \Psi_{j,\epsilon} \partial_x u_n)\,\dxy = \lim\limits_{\epsilon \rightarrow 0}  \limsup\limits_{n \rightarrow \infty} \int_{\R^2} |\partial_x u_n|^2 \Psi_{j,\epsilon}\,\dxy.\]
Thus we obtain
\begin{align*}
    0&=\lim\limits_{n \rightarrow \infty} \langle J'_\kappa(u_n,v_n) |(u_n \Psi_{j, \epsilon},0) \rangle\\
    &= \lim\limits_{n \rightarrow \infty} \( \frac{C(s_1)}{2} \int_{\R^3} \frac{|u_n(x,y)-u_n(x,z)|^2}{|y-z|^{1+2s_1}} \Psi_{j, \epsilon}(x,y)\, \d z \dxy\\
    &\quad + \frac{C(s_1)}{2} \int_{\R^3} \frac{(u_n(x,y)-u_n(x,z)) \cdot (\Psi_{j, \epsilon}(x,y) - \Psi_{j, \epsilon}(x,z)}{|y-z|^{1+2s_1}} u_n(x,z)\, \d z \dxy\\
    &\quad + \int_{\R^2} u_n^2 \Psi_{j, \epsilon} + \int_{\R^2} |\partial_x u_n|^2 \Psi_{j, \epsilon}\, \dxy + \int_{\R^2} \partial_x u_n \partial(\Psi_{j, \epsilon}) u_n\, \dxy\\
    &\quad - \int_{\R^2} |u_n|^{2_{s_1}} \Psi_{j, \epsilon}\, \dxy - \kappa \alpha \int_{\R^2} h(x,y) |u_n|^{\alpha} |v_n|^{\beta} \Psi_{j, \epsilon}\, \dxy \).
\end{align*}
Now taking $\epsilon \rightarrow 0$ in the above equality, we obtain
\begin{align} \label{two one}
    0 = \lim\limits_{\epsilon \rightarrow 0^+} \int_{\R^2} \Psi_{j, \epsilon} \d \mu - \lim\limits_{\epsilon \rightarrow 0^+} \int_{\R^2} \Psi_{j, \epsilon} \d \nu - \lim\limits_{\epsilon \rightarrow 0^+} \lim\limits_{n \rightarrow \infty}\kappa \alpha \int_{\R^2} h(x,y) |u_n|^{\alpha} |v_n|^{\beta} \Psi_{j, \epsilon}\, \dxy.
\end{align}
Notice that $0 \notin supp(\Psi_{j,\epsilon})$ for $\epsilon$ being sufficiently small and also $h \in L^1(\R^2) \cap L^{\infty}(\R^2)$, we deduce that
\begin{align*}
    \int_{\R^2} \Psi_{j, \epsilon} \d \mu &\geq \int_{\R^2} \Psi_{j, \epsilon} ~(\uone )\, \dxy + \sum_{j \in \mathcal{J}} \mu_j \delta_{(x_j,y_j)} \Psi_{j,\epsilon}+ \mu_0 \delta_0(\Psi_{j,\epsilon})\\
    &= \int_{\R^2} \Psi_{j, \epsilon} ~(\uone )\, \dxy + \sum_{j \in \mathcal{J}} \mu_j \delta_{(x_j,y_j)} \Psi_{j,\epsilon}+ \mu_0 \Psi_{j,\epsilon}(0,0).
\end{align*}
Consequently we have
\begin{equation} \label{two two}
    \lim\limits_{\epsilon \rightarrow 0^+}\int_{\R^2} \Psi_{j, \epsilon} \d \mu \geq \mu_j,
\end{equation}
and similarly we obtain
\begin{equation} \label{two three}
    \lim\limits_{\epsilon \rightarrow 0^+}\int_{\R^2} \Psi_{j, \epsilon} \d \nu = \nu_j.
\end{equation}
Observe that both conditions $\alpha+\beta < \min\{2_{s_1}, 2_{s_2}\}$ and $\alpha + \beta = \min\{2_{s_1}, 2_{s_2}\} < \max\{2_{s_1}, 2_{s_2}\}$ imply $\frac{\alpha}{2_{s_1}} + \frac{\beta}{2_{s_2}}<1$.
Now using Holder's inequality, we estimate the following integral
\begin{align*}
    &\int_{\R^2} h(x,y) |u_n|^{\alpha} |v_n|^{\beta} \Psi_{j, \epsilon}\, \dxy\\
    & \quad =  \int_{\R^2} \(h(x,y)\Psi_{j, \epsilon}(x,y)\)^{1-\frac{\alpha}{2_{s_1}} - \frac{\beta}{2_{s_2}}} \(h(x,y)\Psi_{j, \epsilon}(x,y)\)^{\frac{\alpha}{2_{s_1}} + \frac{\beta}{2_{s_2}}} |u_n|^{\alpha} |v_n|^{\beta} \dxy\\
    & \quad \leq \(\int_{\R^2}\(h(x,y)\Psi_{j, \epsilon}(x,y) \dxy \)^{1-\frac{\alpha}{2_{s_1}} - \frac{\beta}{2_{s_2}}} \(\int_{\R^2}\(h(x,y)|u_n|^{2_{s_1}} \Psi_{j, \epsilon}(x,y) \dxy \)^{\frac{\alpha}{2_{s_1}}} \\
    &\quad\quad \(\int_{\R^2}h(x,y)|v_n|^{2_{s_2}} \Psi_{j, \epsilon}(x,y) \dxy \)^{\frac{\beta}{2_{s_2}}}.
\end{align*}
The first integral on the RHS of the above inequality tends to $0$ as $\epsilon \rightarrow 0^+$ and the rest two integrals are bounded as limit $n \rightarrow \infty$. Hence, we can conclude that
\begin{equation} \label{two four}
 \lim\limits_{\epsilon \rightarrow 0^+}\lim\limits_{n \rightarrow \infty}   \int_{\R^2} h(x,y) |u_n|^{\alpha} |v_n|^{\beta} \Psi_{j, \epsilon}\, \dxy=0.
\end{equation}
Now combining \eqref{two one},\eqref{two two}, \eqref{two three} and \eqref{two four} gives
\begin{equation} \label{two five}
    \mu_j \leq \nu_j ~~\text{as   } \epsilon \rightarrow 0.
\end{equation}
From \eqref{one seven} and \eqref{two five}, we have the following
\[ \Lambda_1 \nu_j^{2/2_{s_1}} \leq \mu_j \leq \nu_j\]
which further gives
\begin{align} \label{two six}
    \text{either } \nu_j =0 ~\text{or}~ \nu_j \geq \Lambda_1^{\frac{1+s_1}{2s_1}}~\text{for all } j \in \mathcal{J} \text{ and } \mathcal{J} \text{ is finite}.
\end{align}
Also, we obtain
\begin{align} \label{two seven}
    \text{either } \Bar{\nu}_k =0 ~\text{or}~ \Bar{\nu}_k \geq \Lambda_2^{\frac{1+s_2}{2s_2}}~\text{for all } k \in \mathcal{K} \text{ and } \mathcal{K} \text{ is finite}.
\end{align}
It is already known that
\begin{align*}
    c &= \bigg(\frac{1}{2}  -\frac{1}{\alpha+\beta} \bigg) \|(u_{n},v_{n})\|_{\mathbb{D}}^{2}+ \bigg(\frac{1}{\alpha+\beta} - \frac{1}{2_{s_{1}}} \bigg) \|u_{n}\|_{{2_{s_{1}}}}^{2_{s_{1}}}+\bigg(\frac{1}{\alpha+\beta} - \frac{1}{2_{s_{2}}} \bigg) \|v_{n}\|_{{2_{s_{2}}}}^{2_{s_{2}}} +o_n(1)\\
   & \geq  \bigg(\frac{1}{2}  -\frac{1}{\alpha+\beta} \bigg)\bigg[ \|(u,v)\|_{\mathbb{D}}^{2} + \sum\limits_{j\in \mathcal{J}}\mu_{j} + \mu_{0} + \mu_{\infty} + \sum\limits_{k \in \mathcal{K}} \Bar{\mu}_{k} + 
       \Bar{\mu}_{0} + \Bar{\mu}_{\infty}  \bigg]\\
  & \quad + \bigg(\frac{1}{\alpha+\beta} - \frac{1}{2_{s_{1}}} \bigg) \bigg[\|u\|_{{2_{s_{1}}}}^{2_{s_{1}}} + \sum\limits_{j\in \mathcal{J}}\nu_{j} + \nu_{0} + \nu_{\infty}\bigg] + \bigg(\frac{1}{\alpha+\beta} - \frac{1}{2_{s_{2}}} \bigg) \bigg[\|v\|_{{2_{s_{2}}}}^{2_{s_{2}}} + \sum\limits_{k\in \mathcal{K}}\Bar{\nu}_{j} + \Bar{\nu}_{0} + \Bar{\nu}_{\infty}\bigg]\\
  & \geq \bigg(\frac{1}{2}  -\frac{1}{\alpha+\beta} \bigg) \bigg[ \Lambda_1 \sum_{j} \nu_j^{2/2_{s_1}} + \Lambda_2 \sum_{k} \Bar{\nu}_k^{2/2_{s_2}} + \Lambda_1 \nu_0^{2/2_{s_1}} + \Lambda_1 \nu_\infty^{2/2_{s_1}}+ \Lambda_2 \Bar{\nu}_0^{2/2_{s_2}} + \Lambda_2 \Bar{\nu}_\infty^{2/2_{s_2}} \bigg]\\
  & \quad + \bigg(\frac{1}{\alpha+\beta} - \frac{1}{2_{s_{1}}} \bigg) \bigg[ \sum\limits_{j\in \mathcal{J}}\nu_{j} + \nu_{0} + \nu_{\infty}\bigg] + \bigg(\frac{1}{\alpha+\beta} - \frac{1}{2_{s_{2}}} \bigg) \bigg[ \sum\limits_{k\in \mathcal{K}}\Bar{\nu}_{j} + \Bar{\nu}_{0} + \Bar{\nu}_{\infty}\bigg]. \numberthis \label{two eight}
\end{align*}
If the concentration is considered at the point $(x_j,y_j)$ only, i.e. $\nu_j>0$ and further from \eqref{two eight} we see that
\begin{align*}
    c \geq \( \frac{1}{2}- \frac{1}{\alpha+\beta}\)\bigg[ \Lambda_1 (\Lambda_1^{\frac{1+s_1}{2s_1}})^{\frac{2(1-s_1)}{2(1+s_1)}} \bigg] + \(\frac{1}{\alpha+\beta}- \frac{1}{2_{s_1}} \) \Lambda_1^{\frac{1+s_1}{2s_1}}=\(\frac{1}{2} - \frac{1}{2_{s_1}} \)\Lambda_1^{\frac{1+s_1}{2s_1}} = \frac{s_1}{1+s_1} \Lambda_1^{\frac{1+s_1}{2s_1}}.\numberthis \label{two nine}
\end{align*}
From \eqref{three zero} we conclude that $\nu_j=\mu_j=0$ for all $j \in \mathcal{J}$. Arguing similarly we also have $\Bar{\nu}_k=\Bar{\mu}_k=0$ for all $k \in \mathcal{K}$. Now if $\nu_0 \neq 0$, then
\begin{align*}
    c \geq \(\frac{1}{2} - \frac{1}{\alpha+\beta} \)\Lambda_{1}^{\frac{1+s_1}{2s_1}} + \(\frac{1}{\alpha + \beta} - \frac{1}{2_{s_1}} \)\Lambda_{1}^{\frac{1+s_1}{2s_1}} = \frac{s_1}{1+s_1} \Lambda_{1}^{\frac{1+s_1}{2s_1}},
\end{align*}
which contradicts the assumption \eqref{three zero}. Therefore, $\nu_0 =0$. By the same token, we also get $\Bar{\nu}_0=0$. Arguing in a similar way we also get $\nu_\infty=0=\Bar{\nu}_\infty$. Hence, there exists a subsequence that strongly converges in $L^{2_{s_1}}(\R^2) \times L^{2_{s_2}}(\R^2)$. Furthermore, we have
\begin{align*}
    \|(u_n-u,v_n-v)\|_{\mathbb{D}}^2 &= \langle J'_{\kappa}(u_n,v_n) | (u_n-u,v_n-v) \rangle + o_n(1),
\end{align*}
which infers that the sequence $\{(u_n,v_n)\}$ strongly converge in $\mathbb{D}$ and the (PS)- condition holds. 
\end{proof}

\section{Properties of \texorpdfstring{$\Bar{c}_{\kappa}$}{TEXT}\label{prop of ck} and proof of Theorem \ref{solution for kappa large}}
We are devoted to prove Proposition \ref{prop1} at the first part of this section.

\begin{proof}[Proof to Proposition \ref{prop1}]
	Firstly, we prove (1). Let $0 \leq \kappa_1 < \kappa_2$. We take a bounded (in $\mathbb{D}$) minimizing sequence $\{(u_n,v_n)\}$ to $c_{\kappa_1}$. There exists a unique $\eta_n = \eta_{(u_n,v_n)}$ such that $(\eta_nu_n,\eta_nv_n) \in \mathcal{N}_{\kappa_2}$. Then we have
	\begin{align*}
		o_n(1) + c_{\kappa_1} = J_{\kappa_1}(u_n,v_n) & \geq J_{\kappa_1}(\eta_nu_n,\eta_nv_n) \\ 
		& = J_{\kappa_2}(\eta_nu_n,\eta_nv_n) + (\kappa_2 - \kappa_1)\eta_n^{\alpha+\beta}\int_{\R^2} h(x,y) |u_n|^{\alpha} |v_n|^{\beta}\, \dxy \\
		& \geq J_{\kappa_2}(\eta_nu_n,\eta_nv_n) \geq c_{\kappa_2},
	\end{align*}
where we use \eqref{max} in the first inequality. Sending $n$ to infinity, one gets $c_{\kappa_1} \geq c_{\kappa_2}$. Furthermore, if $c_{\kappa_1}$ is achieved by $(u,v)$, we know $u \neq 0$ and $v \neq 0$. There exists a unique $\eta = \eta_{(u,v)}$ such that $(\eta u,\eta v) \in \mathcal{N}_{\kappa_2}$. Then we have
\begin{align*}
	c_{\kappa_1} = J_{\kappa_1}(u,v) & \geq J_{\kappa_1}(\eta u,\eta v) \\ 
	& = J_{\kappa_2}(\eta u,\eta v) + (\kappa_2 - \kappa_1)\eta^{\alpha+\beta}\int_{\R^2} h(x,y) |u|^{\alpha} |v|^{\beta}\, \dxy \\
	& > J_{\kappa_2}(\eta u,\eta v) \geq c_{\kappa_2}.
\end{align*}
This completes the proof to (1).

Next, we prove (2). By (1), it is sufficient to prove that $c_{\kappa_1} \leq \liminf\limits_{\kappa \to \kappa_1^+}c_{\kappa}$ for all $\kappa_1 \geq 0$. For any sequence $\{\kappa_n\}$ with $\kappa_n \to \kappa_1^+$, we can take a bounded (in $\mathbb{D}$) sequence $\{(u_n,v_n)\}$ such that $J_{\kappa_n}(u_n,v_n) = c_{\kappa_n} + o_n(1)$. There exists a unique $\eta_n > 0$ such that $(\eta_n u,\eta_n v) \in \mathcal{N}_{\kappa_1}$. We claim that $\eta_n$ is bounded. Indeed, if $ \eta_n \to \infty$ up to subsequences, using \eqref{Algebraic equation} (for $\kappa = \kappa_1$) we have $\|u_n\|_{{2_{s_{1}}^{}}} \to 0, \|v_n\|_{{2_{s_{2}}^{}}} \to 0, \int_{\mathbb{R}^{2}} h(x,y)|u_n|^{\alpha}|v_n|^{\beta}\mathrm{d}x \mathrm{d}y \to 0$ as $n \to \infty$. Then using \eqref{Algebraic equation} again (for $\kappa = \kappa_n$) one gets $\|(u_n,v_n)\|_{\mathbb{D}} \to 0$, which is a contradiction. Hence, $\eta_n$ is bounded. Then we have
\begin{align*}
	o_n(1) + c_{\kappa_n} = J_{\kappa_n}(u_n,v_n) & \geq J_{\kappa_n}(\eta_n u_n,\eta_n v_n) \\ 
	& = J_{\kappa_1}(\eta_n u_n,\eta_n v_n) + {(\kappa_n - \kappa_1)}\eta_n^{\alpha+\beta}\int_{\R^2} h(x,y) |u_n|^{\alpha} |v_n|^{\beta}\, \dxy \\
	& = J_{\kappa_1}(\eta_n u_n,\eta_n v_n) + o_n(1) \geq c_{\kappa_1} + o_n(1).
\end{align*}
Sending $n$ to infinity, one gets $\liminf\limits_{n \to \infty}c_{\kappa_n} \geq c_{\kappa_1}$. The proof to (2) is complete.

Finally, we prove (3). We firstly prove that $c_0 = \min\bigg\{ \frac{s_1}{1+s_1} \Lambda_1^{\frac{1+s_1}{2s_1}}, \frac{s_2}{1+s_2} \Lambda_2^{\frac{1+s_2}{2s_2}} \bigg\}$. \eqref{equivalent norm} reads in the case $\kappa = 0$ that for any $(u,v) \in \mathcal{N}_0$,
\begin{align*} 
		\|(u,v)\|_{\mathbb{D}}^{2} 
		= \|u\|_{{2_{s_{1}}^{}}}^{2_{s_{1}}^{}} + \|v\|_{{2_{s_{2}}^{}}}^{2_{s_{2}}^{}}.
\end{align*}
We will consider three cases: $u \neq 0, v = 0$; $u = 0, v \neq 0$;  and $u \neq 0, v\neq 0$. When $u \neq 0, v = 0$, from
$$
\Lambda_{1} \|u\|_{2_{s_1}}^2 \leq \|u\|_{\mathcal{H}^{1,s_1}(\R^2)}^2 = \|u\|_{{2_{s_{1}}^{}}}^{2_{s_{1}}^{}},
$$
we deduce that $\|u\|_{2_{s_1}} \geq \Lambda_{1}^{\frac{1-s_1}{4s_1}}$. Hence, by using \eqref{energy functional on Nehari manifold} one gets
$$
	J_{\kappa}(u,0) = \frac{s_{1}}{1+s_1} \|u\|_{{2_{s_{1}}^{}}}^{2_{s_{1}}^{}} \geq \frac{s_{1}}{1+s_1}\Lambda_{1}^{\frac{1+s_1}{2s_1}}.
$$
Similarly, for any $(0,v) \in \mathcal{N}_0$, we have 
$$
J_{\kappa}(0,v) \geq \frac{s_2}{1+s_2} \Lambda_2^{\frac{1+s_2}{2s_2}}.
$$
When $u \neq 0, v\neq 0$, from
$$
\Lambda_{1} \|u\|_{2_{s_1}}^2 + \Lambda_{2} \|v\|_{2_{s_2}}^2 \leq \|(u,v)\|_{\mathbb{D}}^{2} 
= \|u\|_{{2_{s_{1}}^{}}}^{2_{s_{1}}^{}} + \|v\|_{{2_{s_{2}}^{}}}^{2_{s_{2}}^{}},
$$
we deduce that $\Lambda_{1} \|u\|_{2_{s_1}}^2 \leq \|u\|_{{2_{s_{1}}^{}}}^{2_{s_{1}}^{}}$ or $\Lambda_{2} \|v\|_{2_{s_2}}^2 \leq \|v\|_{{2_{s_{2}}^{}}}^{2_{s_{2}}^{}}$. If the former holds, we have
$$
J_{\kappa}(u,v) = \frac{s_{1}}{1+s_1} \|u\|_{{2_{s_{1}}^{}}}^{2_{s_{1}}^{}} + \frac{s_{2}}{1+s_2} \|v\|_{{2_{s_{2}}^{}}}^{2_{s_{2}}^{}} > \frac{s_{1}}{1+s_1}\Lambda_{1}^{\frac{1+s_1}{2s_1}}.
$$
If the later holds, we have
$$
J_{\kappa}(u,v) = \frac{s_{1}}{1+s_1} \|u\|_{{2_{s_{1}}^{}}}^{2_{s_{1}}^{}} + \frac{s_{2}}{1+s_2} \|v\|_{{2_{s_{2}}^{}}}^{2_{s_{2}}^{}} > \frac{s_2}{1+s_2} \Lambda_2^{\frac{1+s_2}{2s_2}}.
$$
Then we can deduce that $c_0 \geq \min\bigg\{ \frac{s_1}{1+s_1} \Lambda_1^{\frac{1+s_1}{2s_1}}, \frac{s_2}{1+s_2} \Lambda_2^{\frac{1+s_2}{2s_2}} \bigg\}$. On the other hand, WLOG, we assume that $\frac{s_1}{1+s_1} \Lambda_1^{\frac{1+s_1}{2s_1}} \leq \frac{s_2}{1+s_2} \Lambda_2^{\frac{1+s_2}{2s_2}}$. Take a minimizing sequence $\{u_n\}$ to $\Lambda_{1}$, i.e., $\frac{\|u_n\|_{\mathcal{H}^{1,s_1}(\R^2)}^2}{\|u_n\|_{{2_{s_{1}}^{}}}^2} \to \Lambda_{1}$ as $n \to \infty$. By replacing by $\{\eta_n u_n\}$ (also denoted by $\{u_n\}$) for some $\{\eta_n\} \subset \mathbb{R}_+$, we may assume that $\{(u_n,0)\} \in \mathcal{N}_0$. Then
$$
J_{\kappa}(u_n,0) = \frac{s_{1}}{1+s_1} \|u_n\|_{{2_{s_{1}}^{}}}^{2_{s_{1}}^{}} \to \frac{s_{1}}{1+s_1}\Lambda_{1}^{\frac{1+s_1}{2{s_1}}} \text{ as } n \to \infty,
$$
implying that $c_0 \leq \frac{s_{1}}{1+s_1}\Lambda_{1}^{\frac{1+s_1}{2_{s_1}}}$. Hence, we obtain that $c_0 = \min\bigg\{ \frac{s_1}{1+s_1} \Lambda_1^{\frac{1+s_1}{2{s_1}}}, \frac{s_2}{1+s_2} \Lambda_2^{\frac{1+s_2}{ 2{s_2}}} \bigg\}$.

Further, We prove that $\Bar{c}_{\kappa} < \min\bigg\{ \frac{s_1}{1+s_1} \Lambda_1^{\frac{1+s_1}{2s_1}}, \frac{s_2}{1+s_2} \Lambda_2^{\frac{1+s_2}{{\tro 2}s_2}} \bigg\}$ for $\kappa$ large enough. For any $(u,v) \in \mathbb{D}\backslash \{(0,0)\}$, there exists a unique $t = t_{(u,v)}>0$ such that $(tu,tv) \in \mathcal{N}_{\kappa}$ and satisfies the algebraic equation
\begin{align*}
	\|(u,v)\|_{\mathbb{D}}^{2} =  t^{2_{s_{1}} -2} \|u\|_{{2_{s_{1}}}}^{2_{s_{1}}} +  t^{2_{s_{2}} -2} \|v\|_{{2_{s_{2}}}}^{2_{s_{2}}} + \kappa (\alpha+\beta ) t^{\alpha+ \beta -2} \int_{\mathbb{R}^{2}} h(x,y)|u|^{\alpha}|v|^{\beta}\,\dxy.\\
	t^{\alpha+ \beta -2} = \frac{ \|(u,v)\|_{\mathbb{D}}^{2} -  t^{2_{s_{1}} -2} \|u\|_{{2_{s_{1}}}}^{2_{s_{1}}} -  t^{2_{s_{2}} -2} \|v\|_{{2_{s_{2}}}}^{2_{s_{2}}}}{\kappa (\alpha+\beta )\int_{\mathbb{R}^{2}} h(x,y)|u|^{\alpha}|v|^{\beta}\,\dxy}, \hspace{4.4cm}
\end{align*}
Suppose $\kappa \rightarrow + \infty$, then using the assumption $\alpha+\beta >2$ we get $t=t_\kappa \rightarrow 0$. Subsequently, we can write
\begin{align*}
	\lim\limits_{\kappa \rightarrow +\infty}t_{\kappa}^{\alpha+ \beta -2}\kappa &=\lim\limits_{\kappa \rightarrow +\infty} \frac{ \|(u,v)\|_{\mathbb{D}}^{2} -  t_{\kappa}^{2_{s_{1}} -2} \|u\|_{{2_{s_{1}}}}^{2_{s_{1}}} -  t_{\kappa}^{2_{s_{2}} -2} \|v\|_{{2_{s_{2}}}}^{2_{s_{2}}}}{(\alpha+\beta )\int_{\mathbb{R}^{2}} h(x,y)|u|^{\alpha}|v|^{\beta}\,\dxy},\\
	\lim\limits_{\kappa \rightarrow +\infty}t_{\kappa}^{\alpha+ \beta -2}\kappa &= \frac{ \|(u,v)\|_{\mathbb{D}}^{2}}{(\alpha+\beta )\int_{\mathbb{R}^{2}} h(x,y)|u|^{\alpha}|v|^{\beta}\, \dxy},~\text{since}~t_{\kappa} \rightarrow 0~\text{as}~\kappa \rightarrow +\infty .
\end{align*}
Further we check the behaviour of the functional $ J_{\kappa}(t_{\kappa}u,t_{\kappa}v)$ as $\kappa \rightarrow + \infty$. The functional is given by 
\begin{align*}
	&J_{\kappa}(t_{\kappa}u,t_{\kappa}v) \\
	&= \frac{t_{\kappa}^{2}}{2} \|(u,v)\|_{\mathbb{D}}^{2} -\frac{t_{\kappa}^{2_{s_{1}}}}{2_{s_{1}}}\int_{\R^2} |u|^{2_{s_{1}}}\,\dxy 
	-\frac{t_{\kappa}^{2_{s_{2}}}}{2_{s_{2}}}\int_{\R^2} |v|^{2_{s_{2}}}\,\dxy - \kappa t_{\kappa}^{\alpha+\beta} \int_{\R^2}h(x,y)|u|^{\alpha}|v|^{\beta} \,\dxy \\
	&= \frac{t_{\kappa}^{2}}{2}\|(u,v)\|_{\mathbb{D}}^{2} -\frac{t_{\kappa}^{2_{s_{1}}}}{2_{s_{1}}}\int_{\R^2} |u|^{2_{s_{1}}}\, \dxy 
	-\frac{t_{\kappa}^{2_{s_{2}}}}{2_{s_{2}}}\int_{\R^2} |v|^{2_{s_{2}}}\, \dxy - \frac{1}{\alpha+\beta}t_{\kappa}^{2} \|(u,v)\|_{\mathbb{D}}^{2} + \\
	&\quad +\frac{t_{\kappa}^{2_{s_{1}}}}{\alpha+\beta}\int_{\R^2} |u|^{2_{s_{1}}}\, \dxy + \frac{t_{\kappa}^{2_{s_{2}}}}{\alpha+\beta}\int_{\R^2} |v|^{2_{s_{2}}}\, \dxy\\
	&= \bigg( \frac{1}{2} - \frac{1}{\alpha+\beta}\bigg)t_{\kappa}^{2} \|(u,v)\|_{\mathbb{D}}^{2}  + \bigg( \frac{1}{\alpha+\beta}- \frac{1}{2_{s_{1}}}\bigg)t_{\kappa}^{2_{s_{1}}}\int_{\R^2} |u|^{2_{s_{1}}}\, \dxy + \bigg( \frac{1}{\alpha+\beta}- \frac{1}{2_{s_{2}}}\bigg)t_{\kappa}^{2_{s_{2}}}\int_{\R^2} |v|^{2_{s_{2}}}\, \dxy\\
	&= \bigg( \frac{1}{2} - \frac{1}{\alpha+\beta} + A_\kappa(u,v) \bigg)t_{\kappa}^{2} \|(u,v)\|_{\mathbb{D}}^{2}.
\end{align*}
The term $A_\kappa(u,v)$ is given by
\begin{align*}
	A_\kappa(u,v) = \frac{t_{\kappa}^{2_{s_{1}}-2}}{\|(u,v)\|_{\mathbb{D}}^{2}} \bigg( \frac{1}{\alpha+\beta}- \frac{1}{2_{s_{1}}}\bigg)\int_{\R^2} |u|^{2_{s_{1}}}\,\dxy + \frac{t_{\kappa}^{2_{s_{2}}-2}}{\|(u,v)\|_{\mathbb{D}}^{2}} \bigg( \frac{1}{\alpha+\beta}- \frac{1}{2_{s_{2}}}\bigg)\int_{\R^2} |v|^{2_{s_{2}}}\,\dxy.
\end{align*}
Notice that $A_\kappa(u,v) \rightarrow 0$ as $\kappa \rightarrow +\infty$.
Hence, $J_{\kappa}(t_{\kappa}u,t_{\kappa}v) = o_\kappa(1)$ as $\kappa \rightarrow +\infty$ and there exists a $\kappa_0>0$ such that if $\kappa > \kappa_0$, where $\kappa_0$ is sufficiently large, then $J_{\kappa}(t_{\kappa}u,t_{\kappa}v) < \min\bigg\{ \frac{s_1}{1+s_1} \Lambda_1^{\frac{1+s_1}{2s_1}}, \frac{s_2}{1+s_2} \Lambda_2^{\frac{1+s_2}{{ 2}s_2}} \bigg\}$. Thus we obtain 
\begin{align*}
	\Bar{c}_{\kappa} &= \inf\limits_{(u,v) \in \mathcal{N}_{\kappa}} J_{\kappa}(u_{},v_{}) \leq J_{\kappa}(t_{\kappa}u,t_{\kappa}v) < \min\bigg\{ \frac{s_1}{1+s_1} \Lambda_1^{\frac{1+s_1}{2s_1}}, \frac{s_2}{1+s_2} \Lambda_2^{\frac{1+s_2}{{ 2}s_2}} \bigg\}
\end{align*} and hence from the above inequality we get
\begin{align*} \label{four one}
	\Bar{c}_{\kappa} < \min\bigg\{ \frac{s_1}{1+s_1} \Lambda_1^{\frac{1+s_1}{2s_1}}, \frac{s_2}{1+s_2} \Lambda_2^{\frac{1+s_2}{{2}s_2}} \bigg\}.
\end{align*}

Now we know that there exists $\kappa^* \in (0,\infty)$ such that 
$$
\Bar{c}_\kappa = \min\bigg\{ \frac{s_1}{1+s_1} \Lambda_1^{\frac{1+s_1}{2s_1}}, \frac{s_2}{1+s_2} \Lambda_2^{\frac{1+s_2}{2s_2}}  \bigg\}, \forall \kappa \in [0,\kappa^*),
$$
$$
\Bar{c}_\kappa < \min\bigg\{ \frac{s_1}{1+s_1} \Lambda_1^{\frac{1+s_1}{2s_1}}, \frac{s_2}{1+s_2} \Lambda_2^{\frac{1+s_2}{2s_2}}  \bigg\}, \forall \kappa > \kappa^*.
$$
Since $\Bar{c}_\kappa$ is continuous, we have $$\Bar{c}_{\kappa^*} = \min\bigg\{ \frac{s_1}{1+s_1} \Lambda_1^{\frac{1+s_1}{2s_1}}, \frac{s_2}{1+s_2} \Lambda_2^{\frac{1+s_2}{2s_2}}  \bigg\}.$$ The proof is complete.
\end{proof}

In the rest part of this section, we prove Theorem \ref{solution for kappa large}.

\begin{proof}[Proof of Theorem \ref{solution for kappa large}]
	We firstly consider the case that $\kappa > \kappa^*$ and prove (1). By Proposition \ref{prop1} (3), 
	$$
	\Bar{c}_\kappa < \min\bigg\{ \frac{s_1}{1+s_1} \Lambda_1^{\frac{1+s_1}{ 2{s_1}}}, \frac{s_2}{1+s_2} \Lambda_2^{\frac{1+s_2}{ 2{s_2}}}  \bigg\},
	$$
	for any $\kappa > \kappa^*$. Then we use Lemma \ref{PS compactness lemma 1} to ensure the existence of $(u_0,v_0) \in \mathbb{D}$ such that $\Bar{c}_{\kappa} = J_{\kappa}(u_0,v_0)$. Next, we define the function $\psi(t): (0, \infty) \rightarrow \mathbb{R}$ given by $\psi(t) = J_{\kappa}(tu,tv),$ for all $t>0.$ Then
	\begin{align*}    
		\psi(t) &= \frac{t^2 }{2}A_1 - \frac{t^{2_{s_1}}}{2_{s_1}} A_2 - \frac{t^{2_{s_2}}}{2_{s_2}}A_3 - A_4\kappa t^{\alpha + \beta},\\
		\psi'(t) &= A_1t - A_2t^{2_{s_1}-1} - A_3 t^{2_{s_2}-1} - A_4\kappa (\alpha +\beta)t^{\alpha +\beta-1},\\
		\psi''(t) &= A_1 - A_2(2_{s_{1}}-1)t^{2_{s_1}-2} - A_3(2_{s_2}-1) t^{2_{s_2}-2} - A_4\kappa (\alpha +\beta)(\alpha +\beta-1)t^{\alpha +\beta-2},\\
		\psi'''(t) &= - A_2(2_{s_1}-1)(2_{s_1}-2)t^{2_{s_1}-3} - A_3(2_{s_2}-1)(2_{s_2}-2) t^{2_{s_2}-3}\\
		& \hspace{5cm}- A_4\kappa (\alpha +\beta)(\alpha +\beta-1)(\alpha +\beta-2)t^{\alpha +\beta-3},
	\end{align*}
	where $A_1 = \| (u,v)\|_{\mathbb{D}}^{2},~A_2= \| u\|_{2_{s_1}}^{2_{s_1}},~A_3= \| v\|_{2_{s_2}}^{2_{s_2}},~A_4 = \int_{\R^2}h(x,y)|u|^{\alpha}|v|^{\beta}\,\,\dxy.$ Clearly, $A_1,A_2,A_3~\text{and}~A_4$ are non-negative. The condition $\alpha + \beta >2$ implies that $\psi'''(t)<0, \text{ for all } t>0$. Therefore the function $\psi'(t)$ is strictly concave for $t>0$. Also, we have $\lim\limits_{t \rightarrow 0}\psi'(t) = 0$ and $\lim\limits_{t \rightarrow +\infty}\psi'(t) = - \infty$. Moreover, the function $\psi'(t)>0$ for $t>0$ small enough. Hence, $\psi'(t)$ has a unique global maximum point at $t=t_0$ and $\psi'(t)$ has a unique root at $t_1 > t_0$ and $\psi''(t)<0$ for $t>t_0$, in particular $\psi''(t_1)<0$. Also, from the equation \eqref{second order derivative} and $\psi(t)$, we observe that $(tu,tv) \in \mathcal{N}_\kappa$ if and only if $\psi''(t)<0$.\\
	Now we consider the function $(|u_0|,|v_0|) \in \mathbb{D}$, then from the above arguments there exists a unique $t_2>0$ such that $t_2(|u_0|,|v_0|) = (t_2|u_0|,t_2|v_0|) \in \mathcal{N}_{\kappa}$ and $t_2$ satisfies the following algebraic equation
	
	\begin{align*} 
		\|(|u_0|,|v_0|)\|_{\mathbb{D}}^{2} =  t_2^{2_{s_{1}} -2} \|u_0\|_{{2_{s_{1}}}}^{2_{s_{1}}} +  t_2^{2_{s_{2}} -2} \|v_0\|_{{2_{s_{2}}}}^{2_{s_{2}}} + \kappa (\alpha+\beta ) t_2^{\alpha+ \beta -2} \int_{\R^2} h(x,y)|u_0|^{\alpha}|v_0|^{\beta}\,\dxy.
	\end{align*}
	Moreover, the pair $(u_0,v_0) \in \mathcal{N}_{\kappa}$ gives
	\begin{align*}
		\begin{split}
			\|(u_0,v_0)\|_{\mathbb{D}}^{2} 
			= \|u_0\|_{{2_{s_{1}}}}^{2_{s_{1}}} + \|v_0\|_{{2_{s_{2}}}}^{2_{s_{2}}} + \kappa (\alpha + \beta) \int_{\R^2} h(x,y)|u_0|^{\alpha}|v_0|^{\beta}\, \dxy.
		\end{split}
	\end{align*}
	Now from the inequality $\|(|u_0|,|v_0|)\|_{\mathbb{D}}^{} \leq  \|(u_0,v_0)\|_{\mathbb{D}}^{}$, one finds that $t_2 \leq 1.$ Since $(u_0,v_0) \in \mathcal{N}_{\kappa}$ and $(u_0,v_0)$ is the unique maximum point of $\psi(t) = J_{\kappa}(tu_0,tv_0), \forall ~t>t_0$. We can deduce that
	\[\Bar{c}_{\kappa} = J_{\kappa}(u_0,v_0) = \max\limits_{t>t_0}J_{\kappa}(tu_0,tv_0) \geq J_{\kappa}(t_2u_0,t_2v_0) \geq J_{\kappa}(t_2|u_0|,t_2|v_0|) \geq \Bar{c}_{\kappa}. \]  
	Thus, we can assume that $u_0 \geq 0$ and $v_0 \geq 0$ in $\R^2$. If $u_0 \equiv 0$, then $v_0$ satisfies the equation
	\begin{equation} \label{Decouple 2}
		-\partial_{xx}v + (- \Delta )_y^{s}v + v = v^{2_{s_2}-1} \text{  in } \R^2,
	\end{equation}
	which implies $v_0 \equiv 0$, since the above equation has only trivial solution (see \cite[Lemma 2.2]{Hichem2023}). Therefore, we get a contradiction to the fact that $(u_0,v_0) \in \mathcal{N}_k$. So we can conclude that either $u_0$ or $v_0$ is not identically zero.
	
	Notice that the Nehari manifold $\mathcal{N}_\kappa$ consists of all the non-trivial critical points of the functional $J_\kappa$. This implies that $c_\kappa \geq \Bar{c}_\kappa$. On the other hand, $c_\kappa \leq J_\kappa(u_0,v_0) = \Bar{c}_\kappa$. Hence, $c_\kappa = \Bar{c}_\kappa$, and so $(u_0,v_0)$ is a non-negative (constant sign) ground state solution to the system \eqref{main problem}.
	
	Next, we consider the case that $\kappa \in [0,\kappa^*)$ and prove (2). Suppose on the contrary that there exists a $\kappa \in [0,\kappa^*)$ such that $c_\kappa$ has a minimizer. By Proposition \ref{prop1} (1), we have $c_{ \kappa^*} < c_\kappa$, which is a contradiction to Proposition \ref{prop1} (3). The proof is complete.
\end{proof}

	\begin{remark}
		Notice that we have derived the non-negative solutions only and to obtain  a positive solution we require the maximum principle for the mixed type Schrodinger operator.
\end{remark} 

\section{Proofs of Theorem \ref{existence}} \label{proof sec6}

The proof of Proposition \ref{prop2} is similar to the one of Proposition \ref{prop1} and we omit the details.

To prove Theorem \ref{existence}, we firstly give a preliminary compactness result in the critical case. Recall that \eqref{H one} is given as
\begin{equation} 
	0 \leq h \in L^{1}(\R^2) \cap L^{\infty}(\R^2)
	,~h~\text{is not zero, continuous~near}~0~\text{and}~\infty,~\text{and}~h(0,0) = \lim\limits_{|(x,y)| \rightarrow +\infty}h(x,y) = 0. \tag\mathbb{H1} \nonumber
\end{equation}
and the space of radial functions in $\mathbb{D}$ is defined as
\[ \mathbb{D}_r := \mathcal{H}^{1,s}_r(\R^2) \times \mathcal{H}^{1, s}_r(\R^2) = \{(u,v) \in \mathbb{D} : u ~\text{and}~v~\text{are radially symmetric}\}. \]

\begin{lemma}\label{critical with h radial}
If ${\alpha} +{\beta} = 2_s $ and $h$ is a radial function satisfying (\ref{H one}). If $\{(u_{n},v_{n})\} \subset \mathbb{D}_{r}$ is a (PS) sequence for $J_{\kappa}$ at level $c \in \mathbb{R}$ such that c satisfies
\begin{equation} \label{three zero 2}
	c < \frac{s}{1+s} \Lambda_r^{\frac{1+s}{2s}},
\end{equation}
the sequence $\{(u_{n},v_{n})\}$ admits a strongly converging subsequence in $\mathbb{D}$.
\end{lemma}  
\begin{proof}
     Since the functions $\{(u_{n},v_{n})\} \subset \mathbb{D}_{r}$ are radial, having concentrations at points other than $0~\text{or}~\infty$ gives a contradiction to countability of concentration points.
 Now if we want to avoid concentration at the points $0$ and $\infty$, it is sufficient to show that
 \begin{align}\label{three three nine}
     \lim\limits_{\epsilon\rightarrow 0}\limsup\limits_{n \rightarrow +\infty}\int_{\mathbb{R}^{2}}h(x,y)|u_{n}|^{\alpha}|v_{n}|^{\beta}\Psi_{0,\epsilon}(x,y)\,\dxy = 0,
 \end{align}
\begin{align}\label{three four zero}
     \lim\limits_{R\rightarrow +\infty}\limsup\limits_{n \rightarrow +\infty}\int_{\sqrt{x^2 + y^2}>R}h(x,y)|u_{n}|^{\alpha}|v_{n}|^{\beta}\Psi_{\infty,\epsilon}(x,y)\,\dxy = 0,
\end{align}
where the cut-off function $\Psi_{0,\epsilon}$ is centred at $(0,0)$ satisfying (\ref{test function phi at j}) and the cut-off function $\Psi_{\infty,\epsilon}$ supported near $\infty$ satisfying the following for $\mathcal{R}>0$ large enough
\begin{align}\label{phi at infinity}
 \Psi_{\infty,\epsilon} = 0 ~~\text{in}~~B_{\mathcal{R}}(0,0),~~ \Psi_{\infty,\epsilon} = 1 ~~\text{in}~~B^{c}_{\mathcal{R}+1}(0,0),~0\leq \Psi_{\infty,\epsilon} \leq 1 ~\text{and}~|\nabla  \Psi_{\infty,\epsilon}| \leq \frac{4}{\epsilon}.
\end{align}

 For any $s \in (0,1)$, the assumption ${\alpha} +{\beta} = 2_s$ implies that $\frac{\alpha}{2_s} + \frac{\beta}{2_s} = 1$. 
By using the assumption on $h$ in \eqref{H one} and the H\"older's inequality, we get
  \begin{align*}
\int_{\mathbb{R}^{2}}h(x,y)|u_{n}|^{\alpha}|v_{n}|^{\beta}\Psi_{0,\epsilon}(x,y)\,\dxy &= \int_{\mathbb{R}^{2}}(h(x,y) \Psi_{0,\epsilon}(x,y))^{\frac{\alpha}{2_{s}} + \frac{\beta}{2_{s}} }|u_{n}|^{\alpha}|v_{n}|^{\beta}\,\dxy\\
      &\leq \bigg(\int_{\mathbb{R}^{2}}h(x,y)|u_{n}|^{2_{s}}\Psi_{0,\epsilon}(x,y)\,\dxy \bigg)^{\frac{\alpha}{2_{s}}}\\ 
      &~~~~~~~~~ \bigg(\int_{\mathbb{R}^{2}}h(x,y)|v_{n}|^{2_{s}}\Psi_{0,\epsilon}(x,y)\,\dxy \bigg)^{\frac{\beta}{2_{s}}}.
  \end{align*}
Now from (\ref{Concentration compactness}) and (\ref{H one}), we have
\begin{align*}
   \lim\limits_{n \rightarrow +\infty} \int_{\mathbb{R}^{2}}h(x,y)|u_{n}|^{2_{s}}\Psi_{0,\epsilon}(x,y)\,\dxy &= \int_{\R^2}h(x,y)|u_0|^{2_{s}}\Psi_{0,\epsilon}(x,y)\,\dxy + \nu_{0}h(0,0) \\
    & \leq \int_{|(x,y)|\leq \epsilon}h(x,y)|u_0|^{2_{s}} \,\dxy, ~\text{since}~h(0,0)=0.
\end{align*}
and 
\begin{align*}
    \lim\limits_{n \rightarrow +\infty}\int_{\mathbb{R}^{2}}h(x,y)|v_{n}|^{2_{s}}\Psi_{0,\epsilon}(x,y)\,\dxy &= \int_{\R^2}h(x,y)|v_0|^{2_{s}}\Psi_{0,\epsilon}(x,y)\,\dxy + \Bar{\nu}_{0}h(0,0) \\
    & \leq \int_{|(x,y)|\leq \epsilon}h(x,y)|v_0|^{2_{s}} \,\dxy, ~\text{since}~h(0,0)=0.
\end{align*}
By combining the above three inequalities, it holds that
\begin{align*}
     &\lim\limits_{\epsilon\rightarrow 0}\limsup\limits_{n \rightarrow +\infty}\int_{\mathbb{R}^{2}}h(x,y)|u_{n}|^{\alpha}|v_{n}|^{\beta}\Psi_{0,\epsilon}(x,y)\,\dxy \\
     & \leq \lim\limits_{\epsilon \rightarrow 0} \Bigg[ \bigg(\int_{|(x,y)|\leq \epsilon}h(x,y)|u_0|^{2_{s}} \,\dxy \bigg)^{\frac{\alpha}{2_{s}}} \bigg(\int_{|(x,y)|\leq \epsilon}h(x,y)|v_0|^{2_{s}} \,\dxy \bigg)^{\frac{\beta}{2_{s}}}\Bigg] =0.
 \end{align*}
Similarly, we can prove (\ref{three four zero}) by using the assumption that $\lim\limits_{|(x,y)|\rightarrow +\infty} h(x,y) = 0$.

Then using similar arguments to Lemma \ref{PS compactness lemma 1}, we can complete the proof.
\end{proof}

\begin{proof}[Proof of Theorem \ref{existence}]
	By Proposition \ref{prop2} (3), 
	$$
	\Bar{c}_{\kappa,r} < \frac{s}{1+s} \Lambda_r^{\frac{1+s}{2s}},
	$$
	for any $\kappa > \kappa^{**}$. Then we use Lemma \ref{critical with h radial} to ensure the existence of $(u_0,v_0) \in \mathbb{D}_r$ such that $\Bar{c}_{\kappa,r} = J_{\kappa}(u_0,v_0)$. Then similar to the proof of Theorem \ref{solution for kappa large}, we know $(u_0,v_0)$ is a non-negative (constant sign) non-trivial solution to the system \eqref{main problem}. The proof is complete.
\end{proof}

	\section{Proof of Theorem \ref{Non-existence theorem}} \label{pohozaev sec}

In the following we prove the Pohozaev type identity which will be useful to derive the non-existence results to the system \eqref{main problem}.
\begin{lemma} \label{Pohozaev identity}
Assume that $(u,v) \in \mathbb{D}$ is a solution to \eqref{main problem} such that { $xu,yu \in L^2(\R^2)$,} $h \in L^{\infty}(\R^2)\cap C^{1}(\R^2)$, and $\alpha+\beta=2_s$. Then $(u,v)$ satisfies the following necessary conditions:
    \begin{align}\label{Pohozaev P1}
    s  &\int_{\R^2} \bigg( |\partial_{x}u|^2 + |\partial_{x}v|^2+ |(-\Delta)_y^{s/2} u|^2 + |(-\Delta)_y^{s/2} v|^2 \bigg)\, \dxy  \notag\\
    &=s  \int_{\R^2} (|u|^{2_s} + |v|^{2_s})\,\dxy+ \kappa s 2_s \int_{\R^2}  h(x,y)  |u|^{\alpha}\, |v|^{\beta}  \, \dxy - \kappa \int_{\R^2}  \big(yh_y + sx h_x\big) |u|^{\alpha}\, |v|^{\beta}  \, \dxy
\end{align}
and
\begin{align}\label{Pohozaev P2}
    \int_{\R^2} \bigg( |\partial_{x}u|^2 + |\partial_{x}v|^2+& |(-\Delta)_y^{s/2} u|^2 + |(-\Delta)_y^{s/2} v|^2 + |u|^2 + |v|^2 \bigg)\, \dxy  \notag\\
    &= \int_{\R^2}\big(|u|^{2_s}+ |v|^{2_s}\big)\,\dxy+ \kappa 2_s \int_{\R^2}  h(x,y)  |u|^{\alpha}\, |v|^{\beta}  \, \dxy.
\end{align}
\end{lemma}
\begin{proof}
{ Let $\mathcal{F}(u)$ denotes the Fourier transform of $u$ defined by 
\[\mathcal{F}(u)(\xi_1,\xi_2):= \int_{\R^2} u(x,y)e^{-2\pi i (x,y)\cdot (\xi_1,\xi_2)}\, \dxy\]
and $\mathcal{F}^{-1}$ denotes the inverse Fourier transform of u. Then we observe that
\begin{align}\label{poh1}
    x\partial_xu = - \mathcal{F}^{-1}(e^{-2\pi i \xi_1}\mathcal{F}(xu)+ \mathcal{F}u), \quad y\partial_yu = - \mathcal{F}^{-1}(e^{-2\pi i \xi_2}\mathcal{F}(yu)+ \mathcal{F}u).
\end{align}
It follows from \eqref{poh1} that $x\partial_xu,y\partial_yu \in L^2(\R^2)$ using $xu,yu \in L^{2}(\R^2)$.}\\
 We multiply the first equation in \eqref{main problem} by $y \partial_y u$ and then integrate over $\R^2$ to obtain the following:
\begin{align}\label{m1}
    - \int_{\R^2} y\, \partial_{xx}u\, \partial_{y}u\, \dxy &+ \int_{\R^2} (-\Delta)_y^s u\, (y \partial_y u)\, \dxy + \int_{\R^2} y u\,\partial_y u\, \dxy \notag\\
    & = \int_{\R^2} y |u|^{2_s -2}u\, \partial_y u\, \dxy + \kappa \alpha \int_{\R^2} yh(x,y)  |u|^{\alpha -2}u\, |v|^{\beta} \partial_y u\, \dxy.
\end{align}
Now we simplify the each integral in the above equation separately. We use Gauss-Divergence theorem to obtain
\begin{align}\label{m2}
    - \int_{\R^2} y \partial_{xx}u\, \partial_{y}u\, \dxy=  \int_{\R^2} y\partial_{x}u \partial_{xy}u\, \dxy = \frac{1}{2} \int_{\R^2} y\partial_{y}|\partial_{x}u|^2\, \dxy= - \frac{1}{2} \int_{\R^2} |\partial_{x}u|^2\, \dxy.
\end{align}
\begin{align}\label{m3}
    \int_{\R^2} yu \partial_{y}u\, \dxy = \frac{1}{2}\int_{\R^2} y \partial_{y}|u|^2\, \dxy = - \frac{1}{2} \int_{\R^2} |u|^2\,\dxy.\hspace{4cm}
\end{align}
\begin{align}\label{m4}
    \int_{\R^2} y|u|^{2_s-2}u \partial_{y}u\, \dxy = \frac{1}{2_s}\int_{\R^2} y \partial_{y}|u|^{2_s}\, \dxy = - \frac{1}{2_s} \int_{\R^2} |u|^{2_s}\,\dxy.\hspace{3cm}
\end{align}
\begin{align}\label{m5}
    &\int_{\R^2} yh(x,y)  |u|^{\alpha -2}u\, |v|^{\beta} \partial_y u\, \dxy =\quad \frac{1}{\alpha} \int_{\R^2} yh(x,y)  \partial_y|u|^{\alpha}\, |v|^{\beta}  \, \dxy \notag\\
    &\qquad = - \frac{1}{\alpha} \int_{\R^2} (yh_y(x,y) + h(x,y))  |u|^{\alpha}\, |v|^{\beta}  \, \dxy- \frac{\beta}{\alpha} \int_{\R^2} yh(x,y)|v|^{\beta-2}v (\partial_y v)  |u|^{\alpha}   \, \dxy.
\end{align}
Note that $[(-\Delta)_y^{s}, y \partial_y] = 2s(-\Delta)_y^s$ and $(-\Delta)_y^s(y \partial_y u)= 2s(-\Delta)_y^s u + y \partial_y ((-\Delta)_y^s u)$. Using the previous identities, we further evaluate the following integral.
\begin{align}
    \int_{\R^2} (-\Delta)_y^s u\, (y \partial_y u)\, \dxy &= \int_{\R^2}  u\, (-\Delta)_y^s(y \partial_y u)\, \dxy \notag\\
    &= 2s \int_{\R^2}  u (-\Delta)_y^s u\, \dxy + \int_{\R^2}  y u \partial_y(-\Delta)_y^s u\, \dxy \notag\\
    &= (2s-1) \int_{\R^2}  u (-\Delta)_y^s u\, \dxy - \int_{\R^2}  y \partial_y u (-\Delta)_y^s u\, \dxy \notag
\end{align}
From the above, we obtain
\begin{align}\label{m6}
    \int_{\R^2} (-\Delta)_y^s u\, (y \partial_y u)\, \dxy = \frac{2s-1}{2} \int_{\R^2}  u (-\Delta)_y^s u\, \dxy = - \frac{1-2s}{2} \int_{\R^2} |(-\Delta)_y^{s/2} u|^2 \, \dxy.
\end{align}
Combining the equations \eqref{m1}-\eqref{m6}, we get
\begin{align}
   & - \frac{1}{2} \int_{\R^2} |\partial_{x}u|^2\, \dxy - \frac{1-2s}{2} \int_{\R^2} |(-\Delta)_y^{s/2} u|^2 \, \dxy  - \frac{1}{2} \int_{\R^2} |u|^2\,\dxy \notag\\
    &\qquad = - \frac{1}{2_s} \int_{\R^2} |u|^{2_s}\,\dxy - \kappa \int_{\R^2} (yh_y(x,y) + h(x,y))  |u|^{\alpha}\, |v|^{\beta}  \, \dxy - \kappa \beta \int_{\R^2} yh(x,y)|v|^{\beta-2}v (\partial_y v)  |u|^{\alpha}   \, \dxy. \notag
\end{align}
We write the above equation as 
\begin{align}\label{m7}
   &  \frac{1}{2} \int_{\R^2} |\partial_{x}u|^2\, \dxy + \frac{1-2s}{2} \int_{\R^2} |(-\Delta)_y^{s/2} u|^2 \, \dxy  + \frac{1}{2} \int_{\R^2} |u|^2\,\dxy \notag\\
    &\qquad =  \frac{1}{2_s} \int_{\R^2} |u|^{2_s}\,\dxy + \kappa \int_{\R^2} (yh_y(x,y) + h(x,y))  |u|^{\alpha}\, |v|^{\beta}  \, \dxy + \kappa \beta \int_{\R^2} yh(x,y)|v|^{\beta-2}v (\partial_y v)  |u|^{\alpha}   \, \dxy. 
\end{align}
Similarly, by multiplying the second equation in \eqref{main problem} by $y \partial_y v$ and integrating over $\R^2$, we obtain
\begin{align}\label{m8}
   &  \frac{1}{2} \int_{\R^2} |\partial_{x}v|^2\, \dxy + \frac{1-2s}{2} \int_{\R^2} |(-\Delta)_y^{s/2} v|^2 \, \dxy  + \frac{1}{2} \int_{\R^2} |v|^2\,\dxy \notag\\
    &\qquad =  \frac{1}{2_s} \int_{\R^2} |v|^{2_s}\,\dxy + \kappa \int_{\R^2} (yh_y(x,y) + h(x,y))  |u|^{\alpha}\, |v|^{\beta}  \, \dxy + \kappa \alpha \int_{\R^2} yh(x,y)|u|^{\alpha-2}u (\partial_y u)  |v|^{\beta}   \, \dxy. 
\end{align}
On the other hand, when we multiply the first equation in \eqref{main problem} by $x \partial_x u$ and integrating over $\R^2$, we have the following equation
\begin{align}\label{m9}
    - \int_{\R^2} x\, \partial_{xx}u\, \partial_{x}u\, \dxy &+ \int_{\R^2} (-\Delta)_y^s u\, (x \partial_x u)\, \dxy + \int_{\R^2} x u\,\partial_x u\, \dxy \notag\\
    & = \int_{\R^2} x |u|^{2_s -2}u\, \partial_x u\, \dxy + \kappa \alpha \int_{\R^2} xh(x,y)  |u|^{\alpha -2}u\, |v|^{\beta} \partial_x u\, \dxy.
\end{align}
Using the similar arguments as done to obtain \eqref{m7}, we get
\begin{align}\label{m10}
   &  -\frac{1}{2} \int_{\R^2} |\partial_{x}u|^2\, \dxy + \frac{1}{2} \int_{\R^2} |(-\Delta)_y^{s/2} u|^2 \, \dxy  + \frac{1}{2} \int_{\R^2} |u|^2\,\dxy \notag\\
    &\qquad =  \frac{1}{2_s} \int_{\R^2} |u|^{2_s}\,\dxy + \kappa \int_{\R^2} (xh_x(x,y) + h(x,y))  |u|^{\alpha}\, |v|^{\beta}  \, \dxy + \kappa \beta \int_{\R^2} xh(x,y)|v|^{\beta-2}v (\partial_x v)  |u|^{\alpha}   \, \dxy. 
\end{align}
Similarly, multiplying the second equation in \eqref{main problem} by $x \partial_x v$, we obtain
\begin{align}\label{m11}
   &  -\frac{1}{2} \int_{\R^2} |\partial_{x}v|^2\, \dxy + \frac{1}{2} \int_{\R^2} |(-\Delta)_y^{s/2} v|^2 \, \dxy  + \frac{1}{2} \int_{\R^2} |v|^2\,\dxy \notag\\
    &\qquad =  \frac{1}{2_s} \int_{\R^2} |v|^{2_s}\,\dxy + \kappa \int_{\R^2} (xh_x(x,y) + h(x,y))  |u|^{\alpha}\, |v|^{\beta}  \, \dxy + \kappa \alpha \int_{\R^2} xh(x,y)|u|^{\alpha-2}u (\partial_x u)  |v|^{\beta}   \, \dxy. 
\end{align}
Multiplying \eqref{m10} by $s$ and adding into \eqref{m7} gives
\begin{align}\label{m12  P1}
   &  \frac{1-s}{2} \int_{\R^2} |\partial_{x}u|^2\, \dxy + \frac{1-s}{2} \int_{\R^2} |(-\Delta)_y^{s/2} u|^2 \, \dxy  + \frac{1+s}{2} \int_{\R^2} |u|^2\,\dxy \notag\\
    & =  \frac{1+s}{2_s} \int_{\R^2} |u|^{2_s}\,\dxy + \kappa (1+s) \int_{\R^2}  h(x,y)  |u|^{\alpha}\, |v|^{\beta}  \, \dxy + \kappa \int_{\R^2}  (yh_y(x,y) + sx h_x(x,y)) |u|^{\alpha}\, |v|^{\beta}  \, \dxy \notag\\
    &\qquad + \kappa \beta \int_{\R^2} \big(yh(x,y)|v|^{\beta-2}v (\partial_y v) + sxh |v|^{\beta-2}v \partial_x v \big)  |u|^{\alpha}   \, \dxy.
\end{align}
Similarly, multiplying \eqref{m11} by $s$ and adding into \eqref{m8} gives
\begin{align}\label{m13 P2}
   &  \frac{1-s}{2} \int_{\R^2} |\partial_{x}v|^2\, \dxy + \frac{1-s}{2} \int_{\R^2} |(-\Delta)_y^{s/2} v|^2 \, \dxy  + \frac{1+s}{2} \int_{\R^2} |v|^2\,\dxy \notag\\
    & =  \frac{1+s}{2_s} \int_{\R^2} |v|^{2_s}\,\dxy + \kappa (1+s) \int_{\R^2}  h(x,y)  |u|^{\alpha}\, |v|^{\beta}  \, \dxy + \kappa \int_{\R^2}  (yh_y(x,y) + sx h_x(x,y)) |u|^{\alpha}\, |v|^{\beta}  \, \dxy \notag\\
    &\qquad + \kappa \alpha \int_{\R^2} \big(yh(x,y)|u|^{\alpha-2}u (\partial_y u) + sxh |u|^{\alpha-2}u \partial_x u \big)  |v|^{\beta}\, \dxy.
\end{align}
Moreover, by multiplying first and second equation in \eqref{main problem} by $u$ and $v$ respectively, and integrating over $\R^2$ we obtain the following identities:
\begin{align}\label{m14 I1}
      \int_{\R^2} |\partial_{x}u|^2\, \dxy +  \int_{\R^2} |(-\Delta)_y^{s/2} u|^2 \, \dxy  + \int_{\R^2} |u|^2\,\dxy =   \int_{\R^2} |u|^{2_s}\,\dxy + \kappa \alpha \int_{\R^2}  h(x,y)  |u|^{\alpha}\, |v|^{\beta}  \, \dxy
\end{align}
and
\begin{align}\label{m15 I2}
      \int_{\R^2} |\partial_{x}v|^2\, \dxy +  \int_{\R^2} |(-\Delta)_y^{s/2} v|^2 \, \dxy  + \int_{\R^2} |v|^2\,\dxy =   \int_{\R^2} |v|^{2_s}\,\dxy + \kappa \beta \int_{\R^2}  h(x,y)  |u|^{\alpha}\, |v|^{\beta}  \, \dxy.
\end{align}
From \eqref{m14 I1} and \eqref{m12  P1}, we deduce that
\begin{align}\label{m16}
   &  s \int_{\R^2} |\partial_{x}u|^2\, \dxy + s \int_{\R^2} |(-\Delta)_y^{s/2} u|^2 \, \dxy  \notag\\
    & = (1+s) \bigg(\frac{1}{2}-\frac{1}{2_s}\bigg) \int_{\R^2} |u|^{2_s}\,\dxy + \frac{\kappa (1+s) (\alpha-2)}{2} \int_{\R^2}  h(x,y)  |u|^{\alpha}\, |v|^{\beta}  \, \dxy \notag\\
    & \qquad - \kappa \int_{\R^2}  (yh_y(x,y) + sx h_x(x,y)) |u|^{\alpha}\, |v|^{\beta}  \, \dxy \notag\\
    &\qquad - \kappa \beta \int_{\R^2} \big(yh(x,y)|v|^{\beta-2}v (\partial_y v) + sxh |v|^{\beta-2}v \partial_x v \big)  |u|^{\alpha}   \, \dxy.
\end{align}
Further, from \eqref{m15 I2} and \eqref{m13  P2}, we also obtain the following
\begin{align}\label{m17}
   &  s \int_{\R^2} |\partial_{x}v|^2\, \dxy + s \int_{\R^2} |(-\Delta)_y^{s/2} v|^2 \, \dxy  \notag\\
    & = (1+s) \bigg(\frac{1}{2}-\frac{1}{2_s}\bigg) \int_{\R^2} |v|^{2_s}\,\dxy + \frac{\kappa (1+s) (\beta-2)}{2} \int_{\R^2}  h(x,y)  |u|^{\alpha}\, |v|^{\beta}  \, \dxy \notag\\
    & \qquad - \kappa \int_{\R^2}  (yh_y(x,y) + sx h_x(x,y)) |u|^{\alpha}\, |v|^{\beta}  \, \dxy \notag\\
    &\qquad - \kappa \alpha \int_{\R^2} \big(yh(x,y)|u|^{\alpha-2}u (\partial_y v) + sxh |u|^{\alpha-2}u \partial_x u \big)  |v|^{\beta}   \, \dxy.
\end{align}
Adding \eqref{m16} and \eqref{m17} gives
\begin{align}
    s& \bigg( \int_{\R^2} |\partial_{x}u|^2\, \dxy +\int_{\R^2} |\partial_{x}v|^2\, \dxy + \int_{\R^2} |(-\Delta)_y^{s/2} u|^2 \, \dxy +  \int_{\R^2} |(-\Delta)_y^{s/2} v|^2 \, \dxy -\int_{\R^2} |u|^{2_s}\,\dxy -\int_{\R^2} |v|^{2_s}\,\dxy\bigg) \notag\\
    &=\frac{\kappa (1+s) (2_s-4)}{2} \int_{\R^2}  h(x,y)  |u|^{\alpha}\, |v|^{\beta}  \, \dxy -2 \kappa \int_{\R^2}  (yh_y(x,y) + sx h_x(x,y)) |u|^{\alpha}\, |v|^{\beta}  \, \dxy \notag\\
    &\qquad - \kappa \int_{\R^2} h(x,y)|u|^\alpha \big[ y \partial_y(|v|^\beta) + sx \partial_x(|v|^\beta) \big]\, \dxy - \kappa \int_{\R^2} h(x,y)\big[ y \partial_y(|u|^\alpha) + sx \partial_x(|u|^\alpha) \big] |v|^\beta \, \dxy \notag
\end{align}
after simplifying the above expression, we get 
\begin{align}\label{m18}
    s& \bigg( \int_{\R^2} |\partial_{x}u|^2\, \dxy +\int_{\R^2} |\partial_{x}v|^2\, \dxy + \int_{\R^2} |(-\Delta)_y^{s/2} u|^2 \, \dxy +  \int_{\R^2} |(-\Delta)_y^{s/2} v|^2 \, \dxy -\int_{\R^2} |u|^{2_s}\,\dxy -\int_{\R^2} |v|^{2_s}\,\dxy\bigg) \notag\\
    &=\kappa s 2_s \int_{\R^2}  h(x,y)  |u|^{\alpha}\, |v|^{\beta}  \, \dxy - \kappa \int_{\R^2}  \big(yh_y(x,y) + sx h_x(x,y)\big) |u|^{\alpha}\, |v|^{\beta}  \, \dxy.
\end{align}
Hence, the proof is complete.
\end{proof} 
Now we are going to prove the non-existence result using the previous lemma.
\begin{proof}[Proof of Theorem \ref{Non-existence theorem}]
$(a)$ Let $(u,v)$ be a solution to \eqref{main problem} { such that $xu,yu \in L^2(\R^2)$}. We first divide the equation \eqref{Pohozaev P1} by $s$ and subsequently, we subtract \eqref{Pohozaev P1} into \eqref{Pohozaev P2} to get
\begin{align} \label{m19}
    \|u\|_2^2 + \|v\|_2^2 = \frac{\kappa}{s} \int_{\R^2}  \big(yh_y(x,y) + sx h_x(x,y)\big) |u|^{\alpha}\, |v|^{\beta}  \, \dxy. 
\end{align}
 One can notice that $u=v=0$ for $\kappa=0$ i.e., there is only a trivial solution for $\kappa=0$. Moreover, if $h$ is a constant function and $\kappa$ is any real number, from \eqref{m19} we have
\[\|u\|_2^2 + \|v\|_2^2 =0,\]
which implies $u,v$ are both identically zero a.e. Hence, there does not exist any non-trivial solution to \eqref{main problem} in this case.\\
$(b)$ We are given that $h \in L^{\infty}(\R^2) \cap C^1(\R^2)$ satisfies $\kappa xh_x\leq 0, \kappa yh_y \leq 0$. Using the hypotheses on $h$ in \eqref{m19}, we deduce that
\[\|u\|_2^2 + \|v\|_2^2 \leq 0. \]
From the above inequality, we can conclude that $u,v$ are both identically zero almost everywhere. Hence, there does not exist any non-trivial solution to \eqref{main problem} for such $h$.
\end{proof} 
 \begin{example}
The following examples of $h$ give the non-existence of solutions to the system \eqref{main problem}.
\begin{itemize}
\item[(a)] If we consider
\begin{align*}
   \displaystyle  h(x,y) = \begin{cases}
            e^{-\frac{1}{1-x^2-y^2}} & \text{ if } \sqrt{x^2+y^2}<1\\
            0 & \text{ if } \sqrt{x^2+y^2} \geq 1
        \end{cases}
    \end{align*}
then $h \in C_c^\infty(\R^2)$ and $\kappa xh_x\leq 0$, $\kappa yh_y \leq 0$ for $\kappa>0$. Thus by Theorem \eqref{Non-existence theorem}, there does not exist any non-trivial solution {$(u,v)$} to \eqref{main problem} {such that $xu,yu \in L^2(\R^2)$}.
\item[(b)] Moreover, if $\displaystyle h(x,y) = e^{\frac{1}{1+x^2+y^2}}$, then $h \in L^{\infty}(\R^2) \cap C^{\infty}(\R^2)$ with $\kappa xh_x\leq 0$ and $\kappa yh_y\leq 0$ for $\kappa>0$. As a result, no non-trivial solution $(u,v)$ exists to \eqref{main problem} { such that $xu,yu \in L^2(\R^2)$}.
\end{itemize}
  \end{example} 
\section*{Acknowledgments}
RK wants to thank the support of the CSIR fellowship, file no. 09/1125(0016)/2020--EMR--I for his Ph.D. work. TM acknowledges the support of the Start up Research Grant from DST-SERB, sanction no. SRG/2022/000524.
\bibliographystyle{abbrv}

\end{document}